\documentclass{amsart}
\usepackage[a4paper,margin=2.5cm, marginparwidth=1.8cm,
            marginparsep=0.2cm]{geometry}
 
\usepackage[T1]{fontenc}

\usepackage{amssymb, amsmath, amsthm}
\usepackage{enumitem}
\usepackage{mathtools}

\usepackage[dvipsnames]{xcolor}
\usepackage[colorlinks=true]{hyperref}

\usepackage[textsize=tiny]{todonotes}

\usepackage{url}

\usepackage[backend=bibtex,style=alphabetic,sorting=nyt,isbn=false,url=false,doi=true,maxalphanames=10,minalphanames=4,mincitenames=4,maxcitenames=10,minnames=4,maxnames=10,giveninits=true,maxbibnames=99]{biblatex}

\AtEveryBibitem{\clearlist{language}}
\renewbibmacro{in:}{}

\theoremstyle{plain}
\newtheorem{theorem}{Theorem}[section]
\newtheorem{conjecture}[theorem]{Conjecture}
\newtheorem{proposition}[theorem]{Proposition}
\newtheorem{lemma}[theorem]{Lemma}
\newtheorem{corollary}[theorem]{Corollary}

\theoremstyle{definition}
\newtheorem{definition}[theorem]{Definition}
\newtheorem{example}[theorem]{Example}
\newtheorem{remark}[theorem]{Remark}

\newcommand{\SRT}{\mathrm{SRT}}
\newcommand{\ZZ}{\mathbb{Z}}
\newcommand{\QQ}{\mathbb{Q}}

\def\oM{\overline{\mathcal{M}}}

\def\ch{{\mathrm{ch}}}
\def\calS{\mathcal{S}}

\newcommand{\Coeff}[1]{\mathop{\mathrm{Coeff}}\nolimits_{\left[#1\right]}}

\title[Triple Hodge integrals and constant Poisson brackets]{Triple Hodge integrals and constant Poisson brackets for rank-one Dubrovin-Zhang hierarchies}

\author{Xavier Blot}
\address{X.~B.: Korteweg--de Vriesinstituut voor Wiskunde, Universiteit van Amsterdam, Postbus 94248, 1090GE Amsterdam, Nederland}
\email{x.j.c.v.blot@uva.nl}	

\author{Guido Carlet}
\address
{G.~C.: Université Bourgogne Europe, CNRS, IMB UMR 5584, 21000 Dijon, France}
\email{guido.carlet@ube.fr}

\author{Dimitrios Makris}
\address
{D.~M.: Université Bourgogne Europe, CNRS, IMB UMR 5584, 21000 Dijon, France}
\email{dimitrios.makris@ube.fr}

\author{Sergey Shadrin}
\address{S.~S.: Korteweg--de Vriesinstituut voor Wiskunde, Universiteit van Amsterdam, Postbus 94248, 1090GE Amsterdam, Nederland}
\email{s.shadrin@uva.nl}

\begin{document}

\begin{abstract} We prove a conjecture of Dubrovin, Liu, Yang, and Zhang that states that the Poisson bracket of the Dubrovin-Zhang hierarchy of a rank-one cohomological field theory is constant if and only if it is given by the triple Hodge class with the parameters satisfying the Calabi-Yau condition. 
\end{abstract}

\maketitle

\tableofcontents

\section{Introduction}

The theory of Dubrovin and Zhang~\cite{dubrovin2001normalformshierarchiesintegrable} can be extended in order to associate a canonical hierarchy of Hamiltonian evolutionary PDEs to a (partial) cohomological field theory with a unit, see~\cite{BPS1,DLYZ-Hodge-conjecture,BLS-DRDZ,BSS25,blot2026descendantsintegrableobservablescohomological}. 

The cohomological field theories with a unit of rank one have been fully classified by Teleman~\cite{Teleman} and can be described in terms of the so-called $\lambda$ classes, the Chern classes of the Hodge bundles on the moduli spaces of curves~\cite{Mumford}. The Dubrovin-Zhang theory in this particular case has been developed in~\cite{DLYZ-Hodge-conjecture}, where, in particular, the authors formulated a conjecture on the conditions that one has to impose to guarantee that the Poisson bracket of the Dubrovin-Zhang hierarchy is constant in the dependent variables, see~\cite[Example~4.6]{DLYZ-Hodge-conjecture}. The purpose of this paper is to prove this conjecture.

The condition for the Poisson bracket to be constant conjectured in~\cite[Example~4.6]{DLYZ-Hodge-conjecture} is that the underlying cohomological field theory is given by the triple Hodge classes with the parameters satisfying the Calabi-Yau condition. It is indeed a very special family of examples, distinguished, in particular, by KP integrability~\cite{Alexandrov}. Moreover, it is closely related to the theory of the topological vertex~\cite{LLLZ} and is governed by the so-called Mari\~{n}o-Vafa formula~\cite{MarinoVafa}, see e.~g. \cite{LiuLiuZhou,OP,JianZhou,Kramer} and the references therein.
In this case, the Dubrovin-Zhang hierarchy is shown in~\cite{CubicHodge-Loop,Hodge-FVH} to be equivalent to the so-called fractional Volterra hierarchy, defined in~\cite{FVH-definition} and further linked to KP integrability in~\cite{Takasaki}. The Poisson bracket  can then be described explicitly, and it is indeed constant as it was conjectured in~\cite[Example~4.6]{DLYZ-Hodge-conjecture}. The remaining open question is whether these are the only cases when the Poisson bracket is constant, and its affirmative answer is precisely the main result of this paper. 

The idea of the proof comes from the so-called DR/DZ correspondence. Namely, there is an alternative approach to the construction of a Hamiltonian system of flows associated with a (partial) cohomological field theory, the so-called DR hierarchy, originally invented by Buryak~\cite{Buryak-DR} and further developed in~\cite{Buryak-Rossi-recursion,BDGR1,BDGR20}. The two constructions of integrable hierarchies are known to be related by Miura transformations~\cite{BS22,BLS-DRDZ,BSS25,blot2026descendantsintegrableobservablescohomological} and this allows to transfer various properties of the structures of the hierarchies on the Dubrovin-Zhang side to the DR side and vice versa. 
In particular, the Poisson structure on the DR side is just $\partial_x$, and we use the formulas for the Miura transformations given in~\cite{BS22,blot2026descendantsintegrableobservablescohomological} to analyze the bracket on the Dubrovin-Zhang side. 

\subsection{Organization of the paper}

In Sec.~\ref{Sec:Hodge classes} we recall how rank-one cohomological field theories are expressed in terms of Hodge classes and discuss various ways of parametrizing these CohFTs in terms of the symmetric functions. 
In Sec.~\ref{Sec: Conjecture} we recall the conjecture of~\cite[Example~4.6]{DLYZ-Hodge-conjecture} and discuss the part of this conjecture concerning the explicit form of the  Poisson bracket in the constant case; this has been already proved in~\cite{Hodge-FVH} via a relation of the corresponding Dubrovin-Zhang hierarchies with the fractional Volterra hierarchy defined in~\cite{FVH-definition}. 
In Sec.~\ref{Sec: General Miura} we recall and use the results of~\cite{BS22,blot2026descendantsintegrableobservablescohomological} to obtain an explicit formula for the Poisson bracket of the Dubrovin-Zhang hierarchy of general rank $1$ CohFT. 
Finally, in Sec.~\ref{Sec: obstructions to constancy} we analyze an explicit family of coefficients of non-constant terms in the Poisson bracket, which provide enough equations to derive the conjectured conditions for the constancy of the Poisson bracket.

\subsection{Acknowledgements}

X.~B. and S.~S. were supported by the Dutch Research Council, project OCENW.M.21.233.
G.~C. and D.~M. were supported by the EIPHI Graduate School, contract ``ANR-17-EURE-0002'', grant ``METTHHOD''. 

\section{Hodge classes} \label{Sec:Hodge classes}

Consider the Hodge bundle over $\oM_g$~\cite{Mumford}. Its Chern classes are denoted by $\lambda_i$, $i=0,\dots,g$, $\lambda_0=1$, and the full Chern class is denoted by $\Lambda(a) \coloneqq \sum_{i=0}^g \lambda_i a^i$. Let $\ch_i$, $i\geq 0$, be the corresponding Chern characters. Mumford's relations imply that $\Lambda(a) \Lambda(-a) = 1$, $\ch_{2i} = 0$ for $i\geq 1$, and $\ch_{2h-1} = 0$ for $h>g$. Abusing notation, we use the same symbols for the $\lambda_i$ and $\ch_i$ on any $\oM_g$, as well as for their pull-backs to $\oM_{g,n}$, $n\geq 0$, under the map that forgets the marked points. Genera $0$ and $1$ are exceptional since we have to define the classes $\lambda_i$ and $\ch_i$ on $\oM_{0,3}$ and $\oM_{1,1}$, respectively, and then pull-back them to $\oM_{0,n}$, $n\geq 3$, and $\oM_{1,n}$, $n\geq 1$, respectively. 

It is proved in~\cite{Teleman} that all CohFTs of rank one with a flat unit of norm one take the form 
\begin{align} \label{eq:classes-CohFT}
	\mathrm{Cl}_{g,n} = \exp\Big(\sum_{i=1}^\infty (-1)^{i-1} (i-1)! p_i \ch_i\Big),
\end{align}
 where for each $g$ and $n$, $g\geq 0$, $n\geq 0$, $2g-2+n>0$, the latter class is considered on the respective $\oM_{g,n}$.  It can be conveniently rewritten as $\mathrm{Cl}_{g,n} = \prod_{i=1}^\infty \Lambda(r_i)$, where $p_i$, $i\geq 1$, is the power sum function in the ring of the symmetric functions in the variables $r_k$, $k\geq 1$. The infinite product $\prod_{i=1}^\infty \Lambda(r_i)$ has to be considered in the same ring as an expression in the elementary symmetric functions $e_i$, $i\geq 1$ in the variables $r_k$, $k\geq 1$.

Note that since $\ch_i=0$ for even $i$, we have a freedom to choose $p_i$, $i$ even, without affecting the classes $\mathrm{Cl}_{g,n}$ defined in~\eqref{eq:classes-CohFT}. It is convenient for us to consider $p_{2h-1}$, $h\geq 1$, as independent parameters and choose $p_{2h}$, $h\geq 1$, as unique polynomials in $p_{2h-1}$, $h\geq 1$, such that $e_{2k}=0$ for all $k\geq 1$. Explicitly, the corresponding specialization of Newton's identities reads
\begin{align} \label{eq:Newton}
\sum_{k=1}^\infty {e_{2k-1}} t^{2k-1} & = \tanh \sum_{h=1}^\infty \frac{p_{2h-1}}{2h-1} t^{2h-1};
& \sum_{h=1}^\infty \frac{p_{2h}}{2h} t^{2h} & = \log \cosh \sum_{h=1}^\infty \frac{p_{2h-1}}{2h-1} t^{2h-1}.
\end{align}
Alternatively we can consider $e_{2k-1}$, $k\geq 1$, as independent parameters and use these formulas to express all $p_i$, $i\geq 1$, ensuring that $e_{2k}=0$, $k\geq 1$.

By definition~\eqref{eq:classes-CohFT}, the homogeneous components of $\mathrm{Cl}_{g,n}$ in the Chow degree are homogeneous of the same degree in the natural degree $\deg_e$ associated to the variables $e_i$ as generators of the ring of symmetric functions, namely $\deg_e e_i =i$, hence any monomial in $\mathrm{Cl}_{g,n}$ has polynomial dependence in a finite number of variables $e_{2k-1}$, $k\geq1$.  

A natural constraint that we can impose on $\mathrm{Cl}_{g,n}$ is that $e_i =0$ for $i \geq 2m$ for some fixed $m>0$. In such case the CohFT depends only on the parameters $e_1, e_3, \dots, e_{2m-1}$. When expressing the variables $e_i$ and $p_i$ respectively as symmetric functions and power sums in $r_1, r_2, \dots$, the constraint above is equivalent to $r_i = 0$ for $i \geq 2m$ together with $e_{2i}(r_1, \dots, r_{2m-1}) =0$ for $i=1, \dots, m-1$. We have therefore that $\mathrm{Cl}_{g,n} = \prod_{i=1}^{2m-1} \Lambda(r_i)$. 

The case $m=1$ corresponds to the linear Hodge integrals, while for $m=2$ we have the important case of the triple Hodge integrals $\mathrm{Cl}_{g,n} = \Lambda(r_1)\Lambda(r_2)\Lambda(r_3)$ with the constraint $e_2(r_1, r_2, r_3)=0$ where $e_2=r_1r_2+r_1r_3+r_2r_3$. If $e_3 = r_1r_2r_3$ is non-zero, one can present the latter constraint as the Calabi-Yau condition $r_1^{-1} + r_2^{-1} + r_3^{-1}=0$. If $e_3$ vanishes this case reduces to $m=1$ case. 

\section{Conjecture on constancy of the Poisson bracket} \label{Sec: Conjecture}

The following conjecture is posed in~\cite{DLYZ-Hodge-conjecture}. Consider the Dubrovin-Zhang hierarchy associated to the cohomological field theory~\eqref{eq:classes-CohFT}. Let its dependent variable be denoted by $w$, and let $w_d = \partial_x^d w$, $d\geq 0$, be the corresponding derivatives with respect to the spatial variable $x$. Its Poisson structure is given by the operator
\begin{align}
	\mathrm{P} = \partial_x + \sum_{g=1}^\infty \epsilon^{2g} \sum_{i=0}^{2g+1} P_{g,i}(w_0,w_1,w_2,\dots) \partial_x^{i},
\end{align}
where $P_{g,i}\in (\mathcal{A}_w)_{2g+1-i}$. Here $\mathcal{A}_w\coloneqq \mathbb{C}[[w]][w_1,w_2,w_3\dots]$ is the graded ring of differential polynomials in $w$, equipped with the grading given by the differential degree defined by $\deg w = 0$ and $\deg w_i = i$, $i\geq 1$, and by $(\mathcal{A}_w)_{d}$ we denote its subspace of differential degree $d$. As discussed above, $\mathrm{P}$ depends on the parameters $p_{2h-1}$, $h\geq 1$, or, via the specialization of the Newton identities given in~\eqref{eq:Newton}, on $e_{2k-1}$, $k\geq 1$. The results of~\cite{BPS1,BPS-2} imply that $P_{g,0}=0$ for all $g\geq 1$.

\begin{conjecture} \label{conj:DLYZ} (1) The operator $\mathrm{P}$ is constant in $w_0,w_1,w_2,\dots$ if and only if $e_{2h-1} = 0$ for $h\geq 3$. In other words, the cohomological field theory~\eqref{eq:classes-CohFT} can be represented as $\mathrm{Cl}_{g,n} = \Lambda(r_1)\Lambda(r_2)\Lambda(r_3)$, where $e_1=r_1+r_2+r_3$ and $e_3 = r_1r_2r_3$ are arbitrary and $e_2=r_1r_2+r_1r_3+r_2r_3$ vanishes. If $r_1,r_2,r_3\not=0$, one can present the latter condition as the Calabi-Yau condition $r_1^{-1} + r_2^{-1} + r_3^{-1}=0$. 
	
(2) Moreover, for the case $\mathrm{Cl}_{g,n} = \Lambda(r_1)\Lambda(r_2)\Lambda(r_3)$ with $r_1,r_2,r_3$ satisfying the Calabi-Yau condition, the Poisson bracket is given explicitly by
\begin{align} \label{eq:bracket-constant}
	\mathrm{P} = \frac{\partial_x}{\calS\Big(\sqrt{\frac{r_1r_2}{r_3}}\epsilon \partial_x\Big)\calS\Big(\sqrt{\frac{r_2r_3}{r_1}}\epsilon \partial_x\Big)\calS\Big(\sqrt{\frac{r_3r_1}{r_2}}\epsilon \partial_x\Big)},
\end{align}
where $\calS(z)= z^{-1}(e^{z/2} - e^{-z/2})$.
\end{conjecture}

Our main result is that this conjecture holds, see Thm.~\ref{thm:main} below. Note that the integrable hierarchy associated to the triple Hodge integrals satisfying the Calabi-Yau condition is studied in~\cite{Hodge-FVH}, where it is related to the fractional Volterra hierarchy (FVH) defined in~\cite{FVH-definition}. See also~\cite{CubicHodge-Loop} for an important intermediate step, and~\cite{Hodge-GUE} for a special case. In particular, the bracket~\eqref{eq:bracket-constant} indeed follows from the Poisson structure of FVH. However, FVH controls not the natural flows of the Dubrovin-Zhang hierarchy but rather their infinite linear combinations, and considering that there are further differences in formulation of the hierarchy, some comment is still needed. To this end, we have: 

\begin{proposition} \label{prop:FVH} Consider the cohomological field theory $\{\mathrm{Cl}_{g,n} = \Lambda(r_1)\Lambda(r_2)\Lambda(r_3)\}$, where $e_1=r_1+r_2+r_3$ and $e_3 = r_1r_2r_3$ are arbitrary and $e_2=r_1r_2+r_1r_3+r_2r_3$ vanishes. Then the Poisson bracket of its Dubrovin-Zhang hierarchy is given by Eq.~\eqref{eq:bracket-constant}.
\end{proposition}

\begin{proof} The proof is essentially a short guide to the logic presented in~\cite{Hodge-FVH} and \cite{CubicHodge-Loop}, the papers  that contain the required statement. Basically, we have to explain two things: (1) why we can use the bracket of the FVH, which is obtained from the Dubrovin-Zhang hierarchy via a non-invertible change of time variables, and (2) how the notation and conventions match the ones we use in the present paper.  

All ingredients of the integrable system depend algebraically (polynomially after some rescaling) on $r_1,r_2,r_3$, or, more precisely, on $e_1,e_3$, thus it is enough to derive the bracket for the rational values of these parameters. This is done in~\cite[Lemma~3.9]{Hodge-FVH}, up to some renormalization and change of notation that we discuss below. 

Note that \cite[Lemma~3.9]{Hodge-FVH} states that the bracket of the FVH is derived via a quasi-Miura transformation of the bracket of its dispersionless limit, which is, strictly speaking, not quite the same as the Riemann hierarchy. The difference is that instead of the standard flows and times of the Dubrovin-Zhang hierarchy in \emph{op.~cit.} and in~\cite{CubicHodge-Loop} the authors use their infinite linear combinations. In particular, the Dubrovin-Zhang loop equation in~\cite{CubicHodge-Loop} is obtained from the infinite linear combinations of the standard Virasoro constraints (cf.~\cite[Eq.~(2.4)]{CubicHodge-Loop}), so formally it is a weaker system of equations. However, as it is pointed out in~\cite{CubicHodge-Loop}, this doesn't affect the Dubrovin-Zhang loop equation~\cite[Thm.~3.9]{CubicHodge-Loop} and still gives its unique solution, cf.~\cite[Prop.~3.9]{CubicHodge-Loop}. Thus, the quasi-Miura transformation used in \cite[Lemma~3.9]{Hodge-FVH} (it is derived from the topological tau function of FVH, which does satisfy the Virasoro constraints according to \cite[Sec.~4]{Hodge-FVH}) is exactly the one that connects the Dubrovin-Zhang hierarchy of $\{\mathrm{Cl}_{g,n} = \Lambda(r_1)\Lambda(r_2)\Lambda(r_3)\}$ for rational $r_1,r_2,r_3$ to its dispersionless limit, and the obtained bracket is indeed the bracket of the Dubrovin-Zhang hierarchy. 

Now we match the formulas explicitly, in our notation. 
First, notice that in~\cite{Hodge-FVH} the parameter $\epsilon$ in the tau function is rescaled and one has to work with $\tilde \epsilon = \sqrt{r_1r_2r_3} \epsilon$. With this change of convention, $\mathrm{P} = S(r_1^{-1}\tilde \epsilon\partial_x)^{-1}S(r_2^{-1}\tilde \epsilon\partial_x)^{-1}S(r_3^{-1}\tilde \epsilon\partial_x)^{-1}\partial_x$. Then, the dependent coordinate used in~\cite{Hodge-FVH} is equal to
$\tilde u = S(r_2^{-1}\tilde \epsilon\partial_x) S(r_3^{-1}\tilde \epsilon\partial_x) u$ (up to an overall constant factor that we eliminate to preserve the dispersionless limit to be the Riemann hierarchy). The Miura transformed bracket $\mathrm{P}$ is then given by $S(r_1^{-1}\tilde \epsilon\partial_x)^{-1}S(r_2^{-1}\tilde \epsilon\partial_x) S(r_3^{-1}\tilde \epsilon\partial_x)\partial_x$, which is exactly the one derived in~\cite[Lemma~3.9]{Hodge-FVH} (again, up to an overall constant factor that we eliminated for the same reason as above).
\end{proof}

\begin{remark} Once we fix $e_1$ and $e_3$ and assume that all other $e_{2k-1}=0$ for $k\geq 3$, the parameters $r_1,r_2,r_3$ are reconstructed as the roots of the equation $x^3-e_1x-e_3=0$. In the limit $e_3\to 0$ two roots, say $r_2$ and $r_3$, tend to zero, and $r_1$ tend to $e_1$. Thus, with fixed $e_1$ and all other $e_{2k-1}=0$ for $k\geq 2$, the classes $\mathrm{Cl}_{g,n}$ are equal to $\Lambda(-r)$, $r=-e_1$. In this limit, Eq.~\eqref{eq:bracket-constant} takes the following form:
\begin{align} \label{eq:bracket-ILW}
	\mathrm{P} = \frac{\partial_x}{\calS(\sqrt{r}\epsilon \partial_x)^2},
\end{align}
which is the bracket of a form of the intermediate long wave hierarchy, proved to be the Dubrovin-Zhang hierarchy for $\{\mathrm{Cl}_{g,n} = \Lambda(-r)\}$ in~\cite{Buryak-ILW}.

The further limit $r\to 0$ reduces the cohomological field theory to the trivial one, and in this case the Poisson bracket degenerates further to $\mathrm{P} = \partial_x$, which is the bracket of the KdV hierarchy, conjectured in~\cite{Witten-KdV} and proved in~\cite{Kontsevich-KdV} (see also~\cite[Sec.~6]{blot2026descendantsintegrableobservablescohomological} for a new proof and a small survey of existing proofs).
\end{remark}

\section{General formula for the Poisson bracket} \label{Sec: General Miura}

In this Section we use the results of~\cite{blot2026descendantsintegrableobservablescohomological}, based on~\cite{BS22,BLRS,blot2024rootedtreeslevelstructures,BLS-DRDZ,BSS25}, to present an explicit formula for $\mathrm{P}$ in the case of arbitrary choice of parameters $e_{2k-1}$, $k\geq 1$.

\subsection{\texorpdfstring{$B$}{B} classes}

The presentation in this Section follows almost verbatim the one in~\cite[Sec.~3.1]{blot2024rootedtreeslevelstructures}.

Let $\SRT_{g,n,m}$ be the set of stable rooted trees of total genus $g$, with $n$ regular legs $\sigma_1,\dots,\sigma_n$ and $m$ extra legs $\sigma_{n+1},\dots,\sigma_{n+m}$, which we refer to as ``frozen'' legs and must always be attached to the root vertex. For a $T\in \SRT_{g,n,m}$ we use the following notation:
\begin{itemize}
	\item $H(T)$ is the set of half-edges of $T$.
	\item $L(T),L_r(T),L_f(T)\subset H(T)$ are the sets of all, regular, and frozen legs of $T$, respectively. $L(T) = L_r(T)\sqcup L_f(T)$.
	\item $H_e(T)\coloneqq H(T)\setminus L(T)$.
	\item $\iota\colon H_e(T)\to H_e(T)$ is the involution that interchanges the half-edges that form an edge.
	\item $E(T)$ is the set of edges of $T$, $E\cong H_e(T)/\iota$.
	\item $H_+(T)\subset H(T)$ is the set of the so-called ``positive'' half-edges that consists of all regular legs of $T$ and of half-edges in $H(T)\setminus L(T)$ directed away from the root at the vertices where they are attached,
	$H_+(T)\cong E(T)\cup L_{r}(T)$; 
	\item $H_-(T)\subset H(T)$ is the set of the so-called ``negative'' half-edges that consists of all frozen legs of $T$ and of half-edges in $H(T)\setminus L(T)$ directed towards the root at the vertices where they are attached, $H_-(T)\cong E(T)\cup L_{f}(T)$;
	\item $V(T),V_{nr}(T)$ are the sets of vertices and non-root vertices of $T$. 
	\item $v_r\in V(T)$ is the root vertex of $T$; $V(T)=\{v_r(T)\}\sqcup V_{nr}(T)$.
	\item For a $v\in V(T)$, $H(v),H_+(v),H_-(v)$ are all, positive, and negative half-edges attached to $v$, respectively. Obviously, $|H_-(v_r)|=m$ and for any $v\in V_{nr}(T)$ we have $|H_-(v)|=1$.
	\item For a $v\in V(T)$ let $g(v)\in \ZZ_{\geq 0}$ be the genus assigned to $v$. The stability condition means that $2g(v)-2+|H(v)|>0$.
	The genus condition reads $\sum_{v\in V(T)} g(v) = g.$
	\item We say that a vertex or a (half-)edge $x$ is a descendant of a vertex or a (half-)edge $y$ if $y$ is on the unique path connecting $x$ to $v_r$. 
	\item For an $h\in H_+(T)$ let $DL(h)$ be the set of all legs that are descendants to $h$, including $h$ itself. Note that $DL(h)\subseteq L_r(T)$ for any $h\in H_+(T)$ and $DL(l)=\{l\}$ for $l\in L_r(T)$. 
	\item For an $e\in E(T)$ let $DL(e)$ be the set of all legs that are descendants to $e$. Note that $DL(e)\subseteq L_r(T)$ for any $e\in E(T)$. 
	\item For a $v\in V(T)$ let $DL(v)$ be the set of all regular legs that are descendants to $v$. In particular, $DL(v_r) = L_r(T)$. 
	\item For a $v\in V(T)$ let $DV(v)\subset V(T)$ be the subset of all vertices that are descendants of $v$, including $v$ itself. For instance, $DV(v_r) = V(T)$. 
\end{itemize}

\subsubsection{Ring of coefficients}
Consider the polynomial ring $Q\coloneqq \QQ[a_1,\dots,a_n]$ and define $a\colon H_+(T) \to Q$, $a\colon E(T)\to Q$, and $a\colon V(T)\to Q$ (abusing notation we use the same symbol $a$ for all of these maps) by 
\begin{align}
	a(\sigma_i)& \coloneqq a_i, &i=1,\dots,n; 
	& & a(h)& \coloneqq \textstyle\sum_{l\in DL(h)} a(l), & h\in H_+(T); \\
	a(e)& \coloneqq \textstyle\sum_{l\in DL(e)} a(l), & e\in E(T); 
	& & a(v)& \coloneqq \textstyle\sum_{l\in DL(v)} a(l), & v\in V(T).
\end{align} 

\subsubsection{Levels}
We enhance the structure of a stable rooted tree to the so-called leveled stable rooted tree (of genus $g$, with $n$ regular and $m$ frozen legs). Let $T\in\SRT_{g,n,m}$. A function $\ell\colon V(T)\to\ZZ_{\geq 0}$ is called a level function if the following conditions are satisfied:
\begin{itemize}
	\item The value of $\ell$ on the root vertex is zero ($\ell(v_r) = 0$).
	\item If $v'\in DV(v)$ and $v'\not=v$, then $\ell(v')>\ell(v)$. 
	\item There are no empty levels, that is, for any $0\leq i \leq \max \ell(V(T))$ the set $\ell^{-1}(i)$ is non-empty. 
\end{itemize}
A leveled stable rooted tree is a pair $(T,\ell)$. The height of a leveled stable rooted tree $(T,\ell)$ is $\ell(T)\coloneqq \max \ell(V(T))$. Let $\mathcal{L}(T)$ denote the set of level functions on $T$. 

Once we fix a level function $\ell\in \mathcal{L}(T)$, we extend the definition of a genus function $g\colon V(T) \to \ZZ_{\geq 0}$ to the level genus function $g^{\mathrm{lvl}}\colon \{0,\dots,\ell(T)\}\to \ZZ_{\geq 0}$ defined as 
\begin{align}
	g^{\mathrm{lvl}} (l) \coloneqq \sum_{v\in V(T), \ell(v) \leq l} g(v). 
\end{align}
Of course, 	$g^{\mathrm{lvl}} (\ell(T))=g$. 

\subsubsection{Degree}
We enhance the structure of a stable rooted tree $T\in SRT_{g,n,m}$ with a so-called degree function $d\colon V(T) \to \ZZ_{\geq 0}$. 
Let $\mathcal{D}(T)$ denote the set of degree functions on $T$. 

In the presence of a level function $\ell\in \mathcal{L}(T)$, we associate with a degree function $d$  a level degree function $d^{\mathrm{lvl}}\colon \{0,\dots,\ell(T)\}\to \ZZ_{\geq 0}$, where 
\begin{align}
	d^{\mathrm{lvl}} (l) \coloneqq \sum_{v\in V(T), \ell(v) \leq l} d(v) + |\{v\in V(T), \ell(v) \leq l\}|-1. 
\end{align}
Define also $d(T)\coloneqq d^{\mathrm{lvl}} (\ell(T))$.

\subsubsection{The leveled rooted tree classes}
Let $T\in\SRT_{g,n,m}$. Assign to each $v\in V(T)$ the moduli space of curves $\overline{\mathcal{M}}_{g(v),|H(v)|}$, where the first $|H_{+}(v)|$ marked points correspond to the positive half-edges attached to $v$ and ordered in an arbitrary but fixed way and the last $|H_{-}(v)|$ marked points correspond to the negative half-edges attached to $v$ (it is just one half edge for a non-root vertex, and for the root vertex we order them following the order of frozen legs $\sigma_{n+1},\dots,\sigma_{n+m}$). Consider the class associated to each vertex:
\begin{equation}
	\label{eq:def-O(v)}
	\Psi(v) \coloneqq \prod_{i=1}^{|H_+(v)|} (1-a(h_i) \psi_i)^{-1}
\end{equation}
For each $d\geq 0$, define
$\Psi(v)_{d}$ to be the  homogeneous component of degree $d$ of the class $\Psi(v)$ with respect to the Chow grading. This coincides with the  homogeneous component of degree $d$, viewed as a polynomial in the variables $a_1,\dots,a_n$.

\begin{definition}
	For each $(g,n,m)$ such that $2g-2+n+m>0$ define the class
	\begin{equation}
		B_{g,n}^{m}\in R^{*}\left(\overline{\mathcal{M}}_{g,n+m}\right)\otimes_{\mathbb{Q}}Q
	\end{equation}
	as
	\begin{align}\label{eq:B-definition}	B_{g,n}^{m}\coloneqq\sum_{\substack{T\in\SRT_{g,n,m},d\in\mathcal{D}(T),\ell\in\mathcal{L}(T)\\
				\forall i<\ell(T)\colon d^{\mathrm{lvl}}(i)\leq2g^{\mathrm{lvl}}(i)-2+m
			}
		}(-1)^{\ell(T)}\biggl(\prod_{e\in E(T)}a(e)\biggr)(\mathrm{gl}_{T})_{*}\bigotimes_{v\in V(T)}\Psi(v)_{d(v)}.
	\end{align}
	Here 
	\begin{align}
		(\mathrm{gl}_{T})_{*} \colon \bigotimes_{v\in V(T)}R^{*}(\overline{\mathcal{M}}_{g(v),|H(v)|})\otimes_{\mathbb{Q}}Q \to R^{*}(\overline{\mathcal{M}}_{g,n+m})\otimes_{\mathbb{Q}}Q
	\end{align} is the boundary pushforward map. 
\end{definition}
The homogeneity structure of $\Psi(v)$ implies that the class $(B_{g,n}^{m})_d$, which is the homogeneous component of the class  $B_{g,n}^{m}$ of degree $d$ with respect to the Chow grading, 
	is a homogeneous polynomial of  degree $d$ in the variables $a_1,\dots,a_n$.

\subsection{Miura transformation} 

The DR/DZ equivalence as presented in~\cite{BS22,blot2026descendantsintegrableobservablescohomological} implies that the following formula gives the Miura transformation mapping the DZ hierarchy associated to the cohomological field theory~\eqref{eq:classes-CohFT} in the dependent coordinate $w$ to the corresponding DR hierarchy in the dependent coordinate $u$.

\begin{proposition}[Proposition~3.3 $(4)$ in \cite{blot2026descendantsintegrableobservablescohomological}] We have
\begin{align} \label{eq:Miura-0}
    u = w - \partial_x^2 \sum_{g,n=1}^\infty \frac{\epsilon^{2g}}{n!} \sum_{\substack{ d_1,\dots,d_n\geq 0 \\ d_1+\cdots+d_n = 2g-2}} \int_{\oM_{g,n}} \mathrm{Cl}_{g,n} \Coeff{\prod_{i=1}^n a_i^{d_i}}B^0_{g,n} \prod_{i=1}^n w_{d_i}.
\end{align}
\end{proposition}

\begin{remark}
As observed in~\cite[Remark 3.9]{blot2026descendantsintegrableobservablescohomological} (the statement there compares the Miura transformations to the natural and to the
normal coordinates of the DR hierarchy, which coincide for rank one CohFTs,
cf.~\cite[Eq.~(3.70)]{blot2026descendantsintegrableobservablescohomological}),    Eq.~\eqref{eq:Miura-0} can be equivalently rewritten as
\begin{align} \label{eq:Miura-1}
    u = w - \partial_x \sum_{g,n=1}^\infty \frac{\epsilon^{2g}}{n!} \sum_{\substack{ d_1,\dots,d_n\geq 0 \\ d_1+\cdots+d_n = 2g-1}} \int_{\oM_{g,n+1}} \mathrm{Cl}_{g,n+1} \Coeff{\prod_{i=1}^n a_i^{d_i}}B^1_{g,n} \prod_{i=1}^n w_{d_i}.
\end{align}
\end{remark}

A structural property of this Miura transformation is established through the following property of the $B$ class. 

\begin{lemma} \label{lemma:push-forward} For any $g\geq 2$, $n\geq 1$, $d_1,\dots,d_n\geq 0$ such $d_1+\cdots+d_n=2g-2$, we have
\begin{align}
    \pi_*\Coeff{\prod_{i=1}^n a_i^{d_i}a_{n+1}^0} B^0_{g,n+1} = 0,
\end{align}
where $\pi \colon \oM_{g,n+1}\to \oM_{g,n}$ is the morphism that forgets the last marked point. 
\end{lemma}

\begin{remark}
    One possible proof follows the one of~\cite[Prop.~4.3]{BuryakGuereRossi}. The case $d_1+\cdots+d_n=2g-2$ is not considered in~\emph{op.~cit.}, but by essentially the same argument it is possible to prove that $\Coeff{\prod_{i=1}^n a_i^{d_i}a_{n+1}^0} B^0_{g,n+1} = \pi^*\Coeff{\prod_{i=1}^n a_i^{d_i}} B^0_{g,n}$. However, we need a simpler statement, so we present a slightly different proof.
\end{remark}

\begin{proof} Recall that $\pi_* \prod_{i=1}^n (1-b_i \psi_i)^{-1} = \sum_{i=1}^n b_i \prod_{i=1}^n (1-b_i \psi_i)^{-1}$. We are going to use this property and combine leveled stable trees $(T,\ell)$, $T\in\SRT_{g,n+1,0}$, $\ell\in \mathcal{L}(T)$, in groups to manifestly show that the push-forward   $\pi_*\Coeff{\prod_{i=1}^n a_i^{d_i}a_{n+1}^0} B^0_{g,n+1}$ vanishes.

Consider $T\in\SRT_{g,n+1,0}$, $d\in\mathcal{D}(T)$ and $\ell\in\mathcal{L}(T)$ satisfying conditions in the definition of $B^0_{g,n+1}$. Assume that the leg $\sigma_{n+1}$ is attached to a vertex $\tilde v\in V(T)$ such that $2g(\tilde v)-1+|H_+(\tilde v)|>1$, that is, once the push-forward erases the leg $\sigma_{n+1}$, the vertex remains stable. Denote the set of triples $(T,d,\ell)$ like this $S_{n+1}$.

In this case, let $T'$ be the tree with removed $\sigma_{n+1}$, $d'$ be the same as $d$ on all vertices except for $\tilde v$, $d'(\tilde v) = d(\tilde v)-1$, and $\ell'$ be the same as $\ell$. Since $d'(\tilde v)=d(\tilde v)-1$, the degree function $d'$ satisfies the strict inequality $(d')^{\mathrm{lvl}}(i)<2g^{\mathrm{lvl}}(i)-2$ for any level $i\geq \ell'(\tilde v)$. This extra unit will allow us below to insert a new vertex after $\tilde v$ without violating the level conditions. We have
\begin{align} \label{eq:push-forward-1} 
    & \pi_* (-1)^{\ell(T)}\biggl(\prod_{e\in E(T)}a(e)\biggr)(\mathrm{gl}_{T})_{*}\bigotimes_{v\in V(T)}\Psi(v)_{d(v)}\,\Big|_{a_{n+1}=0} =
    \\ \notag & 
    a(v')\cdot (-1)^{\ell'(T')}\biggl(\prod_{e\in E(T')}a(e)\biggr)(\mathrm{gl}_{T'})_{*}\bigotimes_{v\in V(T')}\Psi(v)_{d'(v)}.
\end{align}

Assume that the set $H_+(\tilde v)$ in $T'$ splits as $\{\rho_1,\dots,\rho_k\}\sqcup \{\sigma_{t_1},\dots,\sigma_{t_m}\}$, where $\rho_i\in H_e(T)$ and the edges $(\rho_i, \iota(\rho_i))$ connect $\tilde v$ to the vertex that we denote $\tilde v_i$. 
As we shall see, there are $k+m$ more trees $T_{\rho_1},\dots,T_{\rho_k}$, $T_{\sigma_{t_1}},\dots,T_{\sigma_{t_m}}$, whose associated classes push-forward to the right hand side of Eq.~\eqref{eq:push-forward-1}, with a factor. We will show that adding these contributions yields zero.

The tree $T_{\rho_j}\in \SRT_{g,n+1,0}$ is obtained from $T'$ by replacing the edge $(\rho_j, \iota(\rho_j))$ with two edges connecting $\tilde v$ and $\tilde v_j$ to a new vertex $w$ such that $g(w)=0$ and $H(w)=3$ and $\sigma_{n+1}$ is attached to $w$. We extend $d'$ to $T_{\rho_j}$ by $d'(w)=0$ (it is the unique possible extension). Then, we are interested in the subset $\mathcal{L}_{\rho_j}\subset \mathcal{L}(T_{\rho_j})$ of the set of level functions on $T_{\rho_j}$ that consists of the level functions $\ell$ such that $\ell(v_1)=\ell(v_2)$ for arbitrary two vertices $v_1,v_2\in V(T_{\rho_i})\setminus \{w\}$ if and only if $\ell'(v_1) = \ell'(v_2)$. 
There are precisely $\ell'(\tilde v_j)-1-\ell'(\tilde v)$ level functions $\ell \in \mathcal{L}_{\rho_j}$ such that $\ell(T_{\rho_j})= \ell(T')$ (those where $w$ is inserted in a pre-existing level strictly between the levels of $\tilde v$ and of $\tilde{v}_j$), and $\ell'(\tilde v_j)-\ell'(\tilde v)$ level functions $\ell \in \mathcal{L}_{\rho_j}$ such that $\ell(T_{\rho_j})= \ell(T')+1$, and no further possibilities (those creating a new level occupied by $w$ alone after the level of $\tilde v$ and before the level of $\tilde{v}_j$), thus $|\mathcal{L}_{\rho_j}| = 2\ell'(\tilde v_j)-2\ell'(\tilde v)-1$. Note that for all level functions $\ell \in \mathcal{L}_{\rho_j}$ the degree function $d'$ satisfies the level condition $\forall i<\ell(T_{\rho_j})\colon (d')^{\mathrm{lvl}}(i)\leq 2g^{\mathrm{lvl}}(i)-2$. Then we have:
\begin{align} \label{eq:push-forward-2} 
    & \pi_* \sum_{\ell\in \mathcal{L}_{\rho_j}}(-1)^{\ell(T_{\rho_j})}\biggl(\prod_{e\in E(T_{\rho_j})}a(e)\biggr)(\mathrm{gl}_{T_{\rho_j}})_{*}\bigotimes_{v\in V(T_{\rho_j})}\Psi(v)_{d'(v)} \,\Big|_{a_{n+1}=0} =
    \\ \notag & 
    a(\rho_j) \Big(\big(\ell'(\tilde v_j)-1-\ell'(\tilde v)\big) - \big(\ell'(\tilde v_j)-\ell'(\tilde v)\big)\Big)\cdot (-1)^{\ell'(T')}\biggl(\prod_{e\in E(T')}a(e)\biggr)(\mathrm{gl}_{T'})_{*}\bigotimes_{v\in V(T')}\Psi(v)_{d'(v)}.
\end{align}

The tree $T_{\sigma_{t_j}}\in \SRT_{g,n+1,0}$ is obtained from $T'$ by replacing the leg $\sigma_{t_j}$ with an edge connecting $\tilde v$ to a new vertex $w$ such that $g(w)=0$ and $H(w)=3$ and $\sigma_{n+1}$ and $\sigma_{t_j}$ are attached to $w$. We extend $d'$ to $T_{\sigma_{t_j}}$ by $d'(w)=0$ (it is the unique possible extension). Then, we are interested in the subset $\mathcal{L}_{\sigma_{t_j}}\subset \mathcal{L}(T_{\sigma_{t_j}})$ of the set of level functions on $T_{\sigma_{t_j}}$ that consists of the level functions $\ell$ such that $\ell(v_1)=\ell(v_2)$ for arbitrary two vertices $v_1,v_2\in V(T_{\sigma_{t_j}})\setminus \{w\}$ if and only if $\ell'(v_1) = \ell'(v_2)$. 
There are precisely $\ell'(T')-\ell'(\tilde v)$ level functions $\ell \in \mathcal{L}_{\sigma_{t_j}}$ such that $\ell(T_{\sigma_{t_j}})= \ell(T')$ and $\ell'(T')-\ell'(\tilde v)+1$ level functions $\ell \in \mathcal{L}_{\sigma_{t_j}}$ such that $\ell(T_{\sigma_{t_j}})= \ell(T')+1$, and no further possibilities, thus $|\mathcal{L}_{\sigma_{t_j}}| = 2\ell'(T')-2\ell'(\tilde v)+1$. Note that for all level functions $\ell \in \mathcal{L}_{\sigma_{t_j}}$ the degree function $d'$ satisfies the level condition $\forall i<\ell(T_{\sigma_{t_j}})\colon (d')^{\mathrm{lvl}}(i)\leq 2g^{\mathrm{lvl}}(i)-2$. Then we have:
\begin{align} \label{eq:push-forward-3} 
    & \pi_* \sum_{\ell\in \mathcal{L}_{\sigma_{t_j}}}(-1)^{\ell(T_{\sigma_{t_j}})}\biggl(\prod_{e\in E(T_{\sigma_{t_j}})}a(e)\biggr)(\mathrm{gl}_{T_{\sigma_{t_j}}})_{*}\bigotimes_{v\in V(T_{\sigma_{t_j}})}\Psi(v)_{d'(v)} =
    \\ \notag & 
    a_{t_j} \Big(\big(\ell'(T')-\ell'(\tilde v)\big) - \big(\ell'(T)-\ell'(\tilde v)+1\big)\Big)\cdot (-1)^{\ell'(T')}\biggl(\prod_{e\in E(T')}a(e)\biggr)(\mathrm{gl}_{T'})_{*}\bigotimes_{v\in V(T')}\Psi(v)_{d'(v)}.
\end{align}
Now notice that the sum of the right hand sides of \eqref{eq:push-forward-1}, \eqref{eq:push-forward-2}, and \eqref{eq:push-forward-3} is equal to zero. On the other hand, the graphs that we have on the left hand sides (indexed by $(T,d,\ell) \in S_{n+1}$) exhaust all possibilities of $T\in\SRT_{g,n+1,0}$, $d\in\mathcal{D}(T)$ and $\ell\in\mathcal{L}(T)$ satisfying conditions in the definition of $B^0_{g,n+1}$. Thus, the sum of the left hand sides of \eqref{eq:push-forward-1}, \eqref{eq:push-forward-2}, and \eqref{eq:push-forward-3} over $(T,d,\ell) \in S_{n+1}$ is equal to $\pi_* B^0_{g,n+1} |_{a_{n+1}=0}$. Thus $\pi_* B^0_{g,n+1} |_{a_{n+1}=0} = 0$.
\end{proof} 

Moreover, recall that Chow degree of $\mathrm{Cl}_{g,n}$ is bounded from above by $3g-3+\delta_{g,1}$. Hence,
\begin{lemma} \label{lemma:bound} For any $g\geq 1$, $n\geq 1$, and $d_1,\dots,d_n\geq 0$ such that $d_1+\cdots+d_n = 2g-2$ the intersection number 
$\int_{\oM_{g,n}} \mathrm{Cl}_{g,n} \Coeff{\prod_{i=1}^n a_i^{d_i}}B^0_{g,n}=0$ for $n\geq 2g-1+\delta_{g,1}$.
\end{lemma}

\begin{corollary} \label{cor:u-w} The second summand in \eqref{eq:Miura-0} does not depend explicitly on $w_0$, and $n$ there is bounded by $2g-2$ for any $g\geq 2$. In particular, we can rewrite \eqref{eq:Miura-0} as
\begin{align} \label{eq:Miura-0-v2}
    u & = w - \epsilon^2 \int_{\oM_{1,1}} \mathrm{Cl}_{1,1} \Coeff{a_1^{0}}B^0_{1,1} w_2 
    \\ \notag & \quad 
    -
    \partial_x^2 \sum_{g=2}^\infty \sum_{n=1}^{2g-2} \frac{\epsilon^{2g}}{n!} \sum_{\substack{ d_1,\dots,d_n\geq 1 \\ d_1+\cdots+d_n = 2g-2}} \int_{\oM_{g,n}} \mathrm{Cl}_{g,n} \Coeff{\prod_{i=1}^n a_i^{d_i}}B^0_{g,n} \prod_{i=1}^n w_{d_i}.
\end{align}
\end{corollary}

\begin{proof} The restriction on $n$ is a direct consequence of Lemma~\ref{lemma:bound}. The second summand in \eqref{eq:Miura-0} does not depend explicitly on $w_0$, since for any $g\geq 2$, $n\geq 1$ we have:
\begin{align}
\int_{\oM_{g,n+1}} \mathrm{Cl}_{g,n+1} \Coeff{\prod_{i=1}^n a_i^{d_i}a_{n+1}^0}B^0_{g,n+1} & = \int_{\oM_{g,n+1}} \pi^*\mathrm{Cl}_{g,n} \Coeff{\prod_{i=1}^n a_i^{d_i}a_{n+1}^0}B^0_{g,n+1} \\ \notag & = \int_{\oM_{g,n}} \mathrm{Cl}_{g,n} \pi_*\Coeff{\prod_{i=1}^n a_i^{d_i}a_{n+1}^0}B^0_{g,n+1} = 0 
\end{align}
by Lemma~\ref{lemma:push-forward}, where $\pi\colon \oM_{g,n+1}\to \oM_{g,n}$, while the case $n=0$ yields the term
\begin{equation}
    - \sum_{g\geq 1} \epsilon^{2g}\ \int_{\oM_{g,1}} \mathrm{Cl}_{g,1} \Coeff{a_1^{2g-2}}B^0_{g,1} w_{2g}
\end{equation}
which also does not depend on $w_0$.
\end{proof}

\subsection{Formula for the bracket} One can use the Miura transformation~\eqref{eq:Miura-0-v2} to connect the Poisson bracket $\mathrm{P}$ of the Dubrovin-Zhang hierarchy of $\{\mathrm{Cl}_{g,n}\}$ to the Poisson bracket  of the DR hierarchy, which is simply $\partial_x$. Let
\begin{align}
    L & \coloneqq \sum_{k=0}^\infty \frac{\partial u}{\partial w_k} \partial_x^k = 1-\tilde L,
    \\ \notag 
    \tilde L & \coloneqq \int_{\oM_{1,1}} \mathrm{Cl}_{1,1} \Coeff{a_1^{0}}B^0_{1,1} (\epsilon \partial_x)^2
    \\ \notag & \qquad 
    + \partial_x^2 \circ \sum_{g=2}^\infty \sum_{n=0}^{2g-3} \frac{\epsilon^{2g}}{n!} \sum_{\substack{ d_1,\dots,d_{n+1}\geq 1 \\ d_1+\cdots+d_{n+1} = 2g-2}} \int_{\oM_{g,n+1}} \mathrm{Cl}_{g,n+1} \Coeff{\prod_{i=1}^{n+1} a_i^{d_i}}B^0_{g,n+1} \prod_{i=1}^n w_{d_i} \partial_x^{d_{n+1}},
\end{align}
where we use Eq.~\eqref{eq:Miura-0-v2} and the property that $\sum_{k=0}^\infty \frac{\partial(\partial_x(v))}{\partial w_k} \partial_x^k = \partial_x\circ \sum_{k=0}^\infty \frac{\partial v}{\partial w_k} \partial_x^k$ for any differential polynomial $v$ in $w_{d}$, $d\geq 0$ (see e.g.~\cite[Lem.~1]{BPS-2}). 

The results of~\cite{BS22,BLS-DRDZ,BSS25}, as summarized in~\cite{blot2026descendantsintegrableobservablescohomological}, imply the following proposition.

\begin{proposition} \label{prop:FormulaForP}  We have
\begin{align} \label{eq:FormulaForP}
    L \mathrm{P} L^* = \partial_x; \qquad \mathrm{P} = L^{-1} \partial_x (L^{-1})^*.
\end{align}
\end{proposition}
Recall that, for a differential operator of the form $A= \sum_i a_i \partial_x^i$, its formal adjoint is defined as $A^* = \sum_i (-\partial_x)^i \circ a_i$. 

\begin{remark} \label{rem:geometric-series}
    Note that we can compute $L^{-1}$ as the geometric series, $L^{-1} = \sum_{i=0}^\infty (\tilde L) ^i$.
\end{remark}

\begin{corollary} \label{cor:no-w0}
$\mathrm{P}$ doesn't depend explicitly on $w_0$. 
\end{corollary}

\begin{proof} It is a direct corollary of Cor.~\ref{cor:u-w}.
\end{proof}

Notice that any monomial in $\mathrm{P}$ has homogeneous polynomial dependence on the variables $e_{2h+1}$ with natural degree 
$\deg_e = \frac12 (\deg -1) + \deg_w$,
where $\deg$ denotes the standard degree in $x$-derivatives, including those appearing in $\partial_x$, and $\deg_w$ is the degree in the variables $w_i$. 

\subsection{Example: \texorpdfstring{$\Lambda(-r)$}{Lambda(-r)}} 
In this section we use Eq.~\eqref{eq:FormulaForP} to explicitly compute $\mathrm{P}$ for the cohomological field theory $\{\mathrm{Cl}_{g,n} = \Lambda(-r)\}$, hence reproducing and proving Eq.~\eqref{eq:bracket-ILW}. 

It is more convenient here to use Eq.~\eqref{eq:Miura-1}. Notice that, since the Chow degree of $\Lambda(-r)$ on $\oM_{g,n}$ is at most $g$, for the dimensional reason in Eq.~\eqref{eq:Miura-1} we must have $n=1$  and the only contribution from $\mathrm{Cl}_{g,n}$ to the integrals defining the Miura transformation is $(-r)^g \lambda_g$. Thus we have
\begin{align} \label{eq:tilde-L}
    \tilde L = \sum_{g=1}^\infty (-r\epsilon^2 \partial^2_x)^g \int_{\oM_{g,2}} \lambda_g \Coeff{a_1^{2g-1}} B^1_{g,1}.
\end{align}
Recall Eq.~\eqref{eq:B-definition}. Note that for the bamboo trees (the only ones that contribute to $B^1_{g,1}$) there is a unique level function $\ell$ that measures the distance to the root vertex, so we can specialize~\eqref{eq:B-definition} as
\begin{align}	\label{eq:Bg11}
    B_{g,1}^{1}=\sum_{\substack{T\in\SRT_{g,1,1},\, d\in\mathcal{D}(T)\\
				\forall i<\ell(T)\colon d^{\mathrm{lvl}}(i)\leq2g^{\mathrm{lvl}}(i)-1
			}
		}(-1)^{|E(T)|}\biggl(\prod_{e\in E(T)}a(e)\biggr)(\mathrm{gl}_{T})_{*}\bigotimes_{v\in V(T)}\Psi(v)_{d(v)}.
\end{align}
In the same vein, we define
\begin{align}	B_{g,1}^{1,<}\coloneqq\sum_{\substack{T\in\SRT_{g,1,1},\, d\in\mathcal{D}(T)\\
				\forall i<\ell(T)\colon d^{\mathrm{lvl}}(i)<2g^{\mathrm{lvl}}(i)-1
			}
		}(-1)^{|E(T)|}\biggl(\prod_{e\in E(T)}a(e)\biggr)(\mathrm{gl}_{T})_{*}\bigotimes_{v\in V(T)}\Psi(v)_{d(v)}.
\end{align}

\begin{lemma} We have:
\begin{align} \label{eq:sum-tilde-L}
    \sum_{k=1}^\infty \tilde L^k = \sum_{g=1}^\infty (-r\epsilon^2 \partial_x^2)^g \int_{\oM_{g,2}} \lambda_g \Coeff{a_1^{2g-1}} B^{1,<}_{g,1}.
\end{align}
\end{lemma}

\begin{proof} Notice that the product $\int_{\oM_{g,2}} \lambda_g 
B^1_{g,1} \cdot \int_{\oM_{h,2}} \lambda_h 
B^1_{h,1} \cdot a_1$ can be described, by concatenation of the involved bamboo trees, as 
\begin{align}
    -\int_{\oM_{g+h,2}} \lambda_{g+h} \sum_{\substack{T\in\SRT_{g+h,1,1},\, d\in\mathcal{D}(T)\\
				\forall i<\ell(T)\colon d^{\mathrm{lvl}}(i)\leq2g^{\mathrm{lvl}}(i)-1 \\
                \exists j\colon g^{\mathrm{lvl}}(j) = g, \, d^{\mathrm{lvl}}(j) = 2g-1
			}
		}(-1)^{|E(T)|}\biggl(\prod_{e\in E(T)}a(e)\biggr)(\mathrm{gl}_{T})_{*}\bigotimes_{v\in V(T)}\Psi(v)_{d(v)}.
\end{align}
In other words, there exists a special level (which is the same as a special edge, for the bamboo trees), where $g^{\mathrm{lvl}}(g) = g$, $d^{\mathrm{lvl}}(j) = 2g-1$, and the sign is adjusted by extra $(-1)$. Therefore, 
\begin{align}
& \sum_{k=1}^\infty \tilde L^k =    \sum_{g=1}^\infty (-r\epsilon^2 \partial_x^2)^g \int_{\oM_{g,2}} \lambda_g \times 
\\ \notag &  \Coeff{a_1^{2g-1}} \sum_{\substack{T\in\SRT_{g,1,1},\, d\in\mathcal{D}(T)\\
				\forall i<\ell(T)\colon d^{\mathrm{lvl}}(i)\leq2g^{\mathrm{lvl}}(i)-1
			}
		}(-1)^{|E(T)|}\biggl(\prod_{e\in E(T)}a(e)\biggr)(\mathrm{gl}_{T})_{*}\bigotimes_{v\in V(T)}\Psi(v)_{d(v)} W(T,d),
\end{align}
where $W(T,d)$ is given by a sum of the sign $(-1)^{|J|}$ over all subsets $J$ of $\{j<\ell(T) \,|\, d^{\mathrm{lvl}}(j) = 2g^{\mathrm{lvl}}(j)-1 \}$, that is, over all possible ways to represent $(T,d)$ as a concatenation of smaller trees:
\begin{align}
    W(T,d) \coloneqq \sum_{J \subset \{j<\ell(T) \,|\, d^{\mathrm{lvl}}(j) = 2g^{\mathrm{lvl}}(j)-1 \}} (-1)^{|J|}. 
\end{align}
Notice that $W(T,d) = 0$ if $\{j<\ell(T) \,|\, d^{\mathrm{lvl}}(j) = 2g^{\mathrm{lvl}}(j)-1 \}$ is non-empty. Hence, the only contributing bamboo trees to $\sum_{k=1}^\infty \tilde L^k$ are the ones in $B^{1,<}_{g,1}$.
\end{proof}

\begin{corollary} \label{cor:single-Hodge-formula} We have:
\begin{align}
    \sum_{k=0}^\infty \tilde L^k = 1+ \sum_{g=1}^\infty (-r\epsilon^2 \partial_x^2)^g \int_{\oM_{g,2}} \lambda_g \psi_1^{2g-1} 
    = \frac{1}{\calS(\sqrt{r} \epsilon \partial_x)}.
\end{align}
\end{corollary}
\begin{proof} Indeed, for the dimensional reason the integral in~\eqref{eq:sum-tilde-L} can be non-trivial only if $d^{\mathrm{lvl}}(0) = 2g^{\mathrm{lvl}}(0)-1$, which is only possible for the one-vertex graphs. The second equality is proved in~\cite[Thm.~2]{FaberPanda}.
\end{proof}

Combining Prop.~\ref{prop:FormulaForP}, Rem.~\ref{rem:geometric-series}, and Cor.~\ref{cor:single-Hodge-formula}, we immediately reproduce Eq.~\eqref{eq:bracket-ILW}:

\begin{proposition} 
The Poisson bracket for $\{\mathrm{CL}_{g,n} = \Lambda(-r)\}$ is given by $\partial_x / \calS(\sqrt{r} \epsilon \partial_x)^2$. 
\end{proposition}

\section{Obstructions to constancy} \label{Sec: obstructions to constancy}

Let $\mathrm{P}$ be the Poisson bracket of the Dubrovin-Zhang hierarchy associated with the cohomological field theory~\eqref{eq:classes-CohFT} parametrized by some $e_{2h-1}\in\mathbb{C}$, $h\geq 1$. Recall that, by Corollary~\ref{cor:no-w0}, $\mathrm{P}$ does not depend on $w_0$. 

The goal of this section is to present a proof of the following theorem.

\begin{theorem} \label{thm:main} If $\mathrm{P}$ is constant in $w_i$, $i\geq 1$, then $e_{2h-1} = 0$ for $h\geq 3$.
In other words, if $\mathrm{P}$ is constant in $w_i$, $i\geq 1$, then there exist $r_1,r_2,r_3\in\mathbb{C}$ such that $r_1r_2+r_1r_3+r_2r_3=0$ and $\mathrm{Cl}_{g,n} = \Lambda(r_1)\Lambda(r_2)\Lambda(r_3)$.
\end{theorem}

We consider $\mathrm{P}$ given by Eq.~\eqref{eq:FormulaForP} as operators with the coefficients that are functions of $e_{2h-1}$, $h\geq 1$. More precisely, Eq.~\eqref{eq:FormulaForP} and the homogeneity properties of the classes $B^{0}_{g,n}$ and $\mathrm{Cl}_{g,n}$ imply that 
\begin{align}
    \mathrm{P} = \partial_x + \sum_{g=1}^\infty \sum_{n=0}^{2g-2}  \frac{\epsilon^{2g}}{n!} \sum_{\substack{d_0,\dots,d_n \geq 1\\ d_0+\cdots+d_n=2g+1}} C_{d_0; d_1,\dots,d_n} \prod_{i=1}^n w_{d_i} \partial_x^{d_0},
\end{align}
where $C_{d_0; d_1,\dots,d_n}$ are some polynomials in $e_{2h-1}$, $h\geq 1$, of homogeneous degree $\frac12(d_0+\cdots+d_n) -\frac12 +n$ with respect to the natural grading $\deg_e$ of the ring of symmetric functions.

The idea of the proof is the following. Since $\rm{P}$ is constant, the coefficients $C_{d_0; d_1,\dots,d_n}$ for $n\geq 1$ must vanish, and each vanishing gives an equation for the symmetric functions involved.  To simplify these equations, we use, independently, that Prop.~\ref{prop:FVH} implies that for $n\geq 1$ we have $C_{d_0; d_1,\dots,d_n}|_{e_{2h-1}=0,\, h\geq 3} = 0$. In this way we obtain a triangular system of equation for $e_{2h-1}$, $h\geq 3$, which implies that $e_{2h-1}$, $h\geq 3$, must vanish. 

\begin{example}
Before proceeding with the general argument, let us show that for $\mathrm{P}$ to be constant the parameter $e_5$ needs to vanish. Let $g=3$ and consider the coefficient of $\epsilon^6 w_1^2 \partial_x^5$ in $\mathrm{P}$, which is equal to $\frac 12 C_{5;1,1}$. One can compute it explicitly as $\int_{\oM_{3,3}} \mathrm{Cl_{3,3}} \Coeff{a_1a_2a_3^2} B^0_{3,3}$. The integral is nontrivial only if we take the homogeneous degree $5$ in $\mathrm{Cl_{3,3}}$, which can be computed using Eqs.~\eqref{eq:classes-CohFT} and~\eqref{eq:Newton} and is equal to 
\begin{align}
120 \ch_5 e_5 + \left( 120 \ch_5 + 3 \ch_1^2 \ch_3 \right) e_1^2 e_3 + \Big( 24 \ch_5 + \ch_1^2 \ch_3 + \frac{1}{120} \ch_1^5 \Big) e_1^5.    
\end{align}
Thus the coefficient of $\epsilon^6 w_1^2 \partial_x^5$ in $\mathrm{P}$ is equal to 
\begin{align}\label{eq:coeff-w12d5}
    \int_{\oM_{3,3}} \left(120 \ch_5 e_5 + \left( 120 \ch_5 + 3 \ch_1^2 \ch_3 \right) e_1^2 e_3 + \Big( 24 \ch_5 + \ch_1^2 \ch_3 + \frac{1}{120} \ch_1^5 \Big) e_1^5  \right) \Coeff{a_1a_2a_3^2} B^0_{3,3}.
\end{align}
However, a part of this expression, namely the integral
\begin{align} \label{eq:term without e_5}
    \int_{\oM_{3,3}} \left( \left( 120 \ch_5 + 3 \ch_1^2 \ch_3 \right) e_1^2 e_3 + \Big( 24 \ch_5 + \ch_1^2 \ch_3 + \frac{1}{120} \ch_1^5 \Big) e_1^5 \right) \Coeff{a_1a_2a_3^2} B^0_{3,3}
\end{align}
is the coefficient of $\epsilon^6 w_1^2 \partial_x^5$ in $\mathrm{P}\vert_{e_{2h-1}=0,\, h\geq 3}$. By Prop.~\ref{prop:FVH}, the Poisson bracket $\mathrm{P}\vert_{e_{2h-1}=0,\, h\geq 3}$ is given by Eq.~\eqref{eq:bracket-constant} and in particular is constant. It follows that the integral \eqref{eq:term without e_5} vanishes. Therefore, the coefficient of $\epsilon^6 w_1^2 \partial_x^5$ in $\rm{P}$ given by Eq.~\eqref{eq:coeff-w12d5} vanishes in the general case if and only if the difference between \eqref{eq:coeff-w12d5} and \eqref{eq:term without e_5} vanishes, that is,
$120 e_5 \int_{\oM_{3,3}} \ch_5 \Coeff{a_1a_2a_3^2} B^0_{3,3} = 0$. Since we can prove that $\int_{\oM_{3,3}} \ch_5 \Coeff{a_1a_2a_3^2} B^0_{3,3} \not=0$, a necessary condition for the bracket $\mathrm{P}$ to be constant is $e_5=0$.
\end{example}

\subsection{Linear part of \texorpdfstring{$\mathrm{P}$}{P}} We consider the part of $\mathrm{P}$ that is linear in the variables 
$e_{2h-1}$.

\begin{lemma} We have:
\begin{align}
    \mathrm{P} & = \partial_x + e_1 \cdot 2 \epsilon^2  \int_{\oM_{1,1}} \ch_1 \Coeff{a_1^0} B^{0}_{1,1} \partial_x^3 
    \\ \notag & \quad + \sum_{h=2}^\infty e_{2h-1} (2h-1)!\sum_{\substack {g\geq h, n\geq 0 \\ g+n = 2h-1}} \frac{\epsilon^{2g}}{n!} \sum_{\substack {d_1,\dots,d_{n+1}\geq 1 \\ d_1+\cdots+d_{n+1} = 2g-2}}\int_{\oM_{g,n+1}} \ch_{2h-1} \Coeff{\prod_{i=1}^{n+1}a_i^{d_i}} B^0_{g,n+1} \times
    \\ \notag & \qquad \qquad \partial_x^2 \circ \bigg( \prod_{i=1}^n w_{d_i} \partial_x^{d_{n+1}-1} + (-1)^{d_{n+1}} \partial_x^{d_{n+1} -1} \circ \prod_{i=1}^n w_{d_i} \bigg) \circ \partial_x^2
    \\ \notag & \quad + O(e_{\bullet} ^2).
\end{align}
\end{lemma}

\begin{proof} We start from $\mathrm{P}$ given by Eq.~\eqref{eq:FormulaForP}. By Newton identities~\eqref{eq:Newton} reduced to relations between $e_{2h-1}$ and $p_{2k-1}$, $h,k\geq 1$, we have 
$p_{2k-1} = (2k-1) e_{2k-1} + O(e_{\bullet} ^2)$, $k\geq 1$ where the rest in this expression is a polynomial in $e_{2h-1}$, $h<k$ of homogeneous degree $2k-1$ with respect to the natural grading of the ring of symmetric functions. Thus we can just consider the linear part of $\{\rm{Cl}_{g,n}\}$ in $p_{2h-1}$, which is equal to $\sum_{h=1}^\infty (2h-2)! p_{2h-1} \ch_{2h-1}$, and then formally replace $p_{2h-1}$ by $(2h-1) e_{2h-1}$, $h\geq 1$. The result follows. 
\end{proof}

Let us write the coefficient of $(2h-1)! e_{2h-1}$, $h\geq 2$, in $\mathrm{P}$, namely
\begin{align}
   & \sum_{\substack {g\geq h, n\geq 0 \\ g+n = 2h-1}} \frac{\epsilon^{2g}}{n!} \sum_{\substack {d_1,\dots,d_{n+1}\geq 1 \\ d_1+\cdots+d_{n+1} = 2g-2}}\int_{\oM_{g,n+1}} \ch_{2h-1} \Coeff{\prod_{i=1}^{n+1}a_i^{d_i}} B^0_{g,n+1} \times
    \\ \notag & \qquad \qquad \partial_x^2 \circ \bigg( \prod_{i=1}^n w_{d_i} \partial_x^{d_{n+1}-1} + (-1)^{d_{n+1}} \partial_x^{d_{n+1} -1} \circ \prod_{i=1}^n w_{d_i} \bigg) \circ \partial_x^2
\end{align}
as the finite sum $\sum_{g=h}^{2h-1}\epsilon^{2g} P_{h,2g+1}$  where $P_{h,2g+1}$ is a homogeneous operator of degree $2g+1$. For $g = 2h-1$ we have that $P_{h, 4h-1}$ is the constant homogeneous operator of degree $4h-1$ given by
\begin{equation}
    P_{h,4h-1} = 
     2 \int_{\oM_{2h-1,1}} \ch_{2h-1} \Coeff{a_1^{4h-4}} B^0_{2h-1,1} \partial_x^{4h-1},
\end{equation}
while for $g=h,\dots, 2h-2$, $P_{h,2g+1}$ is the following homogeneous operator of order at most $3g-2h+2$
\begin{align}
     P_{h,2g+1}& =\frac{1}{(2h-g-1)!} \sum_{\substack {d_1,\dots,d_{2h-g}\geq 1 \\ d_1+\cdots+d_{2h-g} = 2g-2}}\int_{\oM_{g,2h-g}} \ch_{2h-1} \Coeff{\prod_{i=1}^{2h-g}a_i^{d_i}} B^0_{g,2h-g} \times
    \\  & \qquad  \partial_x^2 \circ \bigg( \prod_{i=1}^{2h-g-1} w_{d_i} \partial_x^{d_{2h-g}-1} + (-1)^{d_{2h-g}} \partial_x^{d_{2h-g} -1} \circ \prod_{i=1}^{2h-g-1} w_{d_i} \bigg) \circ \partial_x^2 .\notag 
\end{align} 
Using the Leibnitz rule $\partial^k_x f=\sum_{j=0}^k \binom{k}{j}f^{(j)} \partial_x^{k-j}$, we can re-write $P_{h,2g+1}$ in the normal form $P_{h,2g+1} = \sum_{l=1}^{3g-2h+2}P_{h,2g+1; l}~ \partial_x^l $ for certain homogeneous differential polynomials $P_{h,2g+1;l}$ of differential degree $2g+1-l$. Below, we provide explicit formulas for those differential polynomials in a simple case.
\begin{example} \label{ex:key-terms}
    For $g=h\geq2$, the operator
    \begin{align}
         &P_{h,2h+1} = \frac{1}{(h-1)!}\sum_{d_h=1}^{h-1} \sum_{\substack{d_1,\dots,d_{h-1}\geq 1  \\ \sum_{j=1}^{h-1}d_j=2h-2-d_{h}}
   } \int_{\oM_{h,h}} \ch_{2h-1} \Coeff{\prod_{i=1}^{h}a_i^{d_i}} B^0_{h,h} \times
    \\ & \qquad \qquad \partial_x^2 \circ \bigg( \prod_{i=1}^{h-1} w_{d_i} \partial_x^{d_{h}-1} + (-1)^{d_{h}} \partial_x^{d_{h} -1} \circ \prod_{i=1}^{h-1} w_{d_i} \bigg) \circ \partial_x^2 \notag 
    \end{align}
    has degree $2h+1$ and order at most $h+2$. Putting it in the normal form $P_{h,2h+1}=\sum_{l=0}^{h+2}P_{h,2h+1;l} \partial_x^l$ we obtain the coefficients
    \begin{align}
        &P_{h,2h+1;h+2} = \frac{1}{(h-1)!} (1+(-1)^{h-1}) 
        \int_{\oM_{h,h}} \ch_{2h-1} \Coeff{\prod_{i=1}^{h-1}a_i~a_h^{h-1}} B^0_{h,h} 
        w_1^{h-1} 
    \end{align}
    and
\begin{align}
        P_{h,2h+1;h+1} & 
         =  \frac{1}{(h-2)!} \bigg[ (2+(-1)^{h-1}h)\int_{\oM_{h,h}} \ch_{2h-1} \Coeff{\prod_{i=1}^{h-1}a_i~a_h^{h-1}} B^0_{h,h} \\
        &\qquad\qquad + ( 1+(-1)^{h}) \int_{\oM_{h,h}} \ch_{2h-1} \Coeff{\prod_{i=1}^{h-2}a_i~a_{h-1}^2~a_h^{h-2}} B^0_{h,h} \bigg] w_1^{h-2}w_2. \notag
\end{align}
\end{example}

\subsection{Non-trivial coefficients}

Our next goal is to compute the first nontrivial coefficients in the expression above presented in Example~\ref{ex:key-terms}. We have

\begin{lemma} \label{lemma:non-vanishing} For odd $h\geq 3$, the coefficient of $(2h-1)!e_{2h-1} \epsilon^{2h} w_1^{h-1} \partial_x^{h+2}$ in $\mathrm{P}$, which is equal to
\begin{align} \label{eq:Odd-h-nonvanishing}
    \frac{1+(-1)^{h-1}}{(h-1)!} \int_{\oM_{h,h}} \ch_{2h-1} \Coeff{a_1\cdots a_{h-1} a_h^{h-1}} B^0_{h,h} 
    \overset{\text{odd\ }h}{=} \frac{\mathrm{B}_{2h} (h-1)!}{2^{h-1}\cdot\left(2h\right)!}
\end{align}
does not vanish ($\mathrm{B}_{2h}$ is the Bernoulli number). For even $h\geq 3$ this coefficient obviously vanishes, but the next to the top one, which is the coefficient of 
$(2h-1)!e_{2h-1} \epsilon^{2h} w_1^{h-2}w_2 \partial_x^{h+1}$ and is equal to
\begin{align} \label{eq:Even-h-nonvanishing}
    & \frac{2+(-1)^{h-1}h}{(h-2)!} \int_{\oM_{h,h}} \ch_{2h-1} \Coeff{a_1\cdots a_{h-1} a_h^{h-1}} B^0_{h,h}
    \\ \notag & \quad 
    + \frac{1+(-1)^{h}}{(h-2)!} \int_{\oM_{h,h}} \ch_{2h-1} \Coeff{a_1\cdots a_{h-2} a_{h-1}^2 a_h^{h-2}} B^0_{h,h} 
    \overset{\text{even\ }h}{=}   
       \frac{-\mathrm{B}_{2h} h! (h-2)}{3\cdot 2^{h-1}\cdot\left(2h\right)!},   
\end{align}
does not vanish.
\end{lemma}

In order to prove this Lemma, we need to develop some computational tools to work efficiently with the integrals $\int_{\oM_{g,n}} \ch_{2g-1} B^0_{g,n}$, which are homogeneous polynomials of degree $g-2+n$ in $a_1,\dots,a_n$. We collect the necessary results in Sec.~\ref{subsubsec:Preliminary} and then provide the computations that prove the lemma in Sec.~\ref{subsubsec:ProofLemmaNonvanishing}.

\subsubsection{Preliminary results} \label{subsubsec:Preliminary}

Here we collect a number of observations that allow us to efficiently compute the values of the intersection numbers in Lemma~\ref{lemma:non-vanishing}. We first obtain, in Lemma~\ref{lem:B-to-Psi}, an explicit expression for the integrals $\int_{\oM_{h,n}} \ch_{2h-1} B^0_{h,n}$, and then use it to establish in Lemma~\ref{lem:Quasi-Dilaton} a dilaton-like property for these integrals.

\begin{lemma} \label{L-5-6}
For any $h\geq 2$ and $n\leq h$ we have:
    \begin{align} \label{eq:reduction-to-simple-trees}
       \int_{\oM_{h,n}} \ch_{2h-1} B^0_{h,n} = & \sum_{T\in\SRT_{h,n,0} } (-1)^{|E(T)|
		}
        \int_{\oM_{h,|H_+(v_r)|}} \ch_{2h-1}  \Psi(v_r)
        \prod_{v\in V_{nr}(T)} a(v)^{|H_+(v)|-1}.
    \end{align}
\end{lemma}

\begin{proof} Note that
\begin{align}
    	 & \int_{\oM_{h,n}} \ch_{2h-1} B^0_{h,n} = \int_{\oM_{h,n}} \ch_{2h-1} \big( B^0_{h,n} \big)_{h+n-2}
         \\
         \notag & = \int_{\oM_{h,n}} \ch_{2h-1} \sum_{\substack{T\in\SRT_{h,n,0},\, d\in\mathcal{D}(T),\, \ell\in\mathcal{L}(T)\\
				\forall i<\ell(T)\colon d^{\mathrm{lvl}}(i)\leq2g^{\mathrm{lvl}}(i)-2 \\ d^{\mathrm{lvl}}(\ell(T)) = h+n-2 }
		}(-1)^{\ell(T)}\biggl(\prod_{e\in E(T)}a(e)\biggr)(\mathrm{gl}_{T})_{*}\bigotimes_{v\in V(T)}\Psi(v)_{d(v)}.
\end{align}
Mumford's relations~\cite{Mumford} imply $(\mathrm{gl}_{T})^{*}\ch_{2h-1}$ is nontrivial if and only if the genus of one of the vertices of $T$ is $h$ and the genera of all other vertices is $0$. The level condition on the root vertex $d^{\mathrm{lvl}}(0)\leq 2g^{\mathrm{lvl}}(0)-2$ implies that the genus $h$ vertex is the root. For the integral to be non-trivial, for dimensional reasons, we must have $d(v_r) = h+|H_+(v_r)|-2$, and $d(v) = |H_+(v)|-2$ for all other vertices $v\not= v_r$. This fixes the unique possible $d\in\mathcal{D}(T)$. Then since $g^{\mathrm{lvl}}(i)=h$ for all $i=0,\dots,\ell(T)$, the inequalities $d^{\mathrm{lvl}}(i)\leq2g^{\mathrm{lvl}}(i)-2$ are satisfied automatically for all $i\leq \ell(T)$ for every $\ell\in\mathcal{L}(T)$. Thus, we can rewrite this expression as
\begin{align}
        	 \int_{\oM_{h,n}} \ch_{2h-1} B^0_{h,n} 
         & = \sum_{\substack{T\in\SRT_{h,n,0}}
		}\biggl(\prod_{e\in E(T)}a(e)\biggr)
        \int_{\oM_{h,|H_+(v_r)|}} \ch_{2h-1}  \Psi(v_r)
        \times
        \\ \notag 
        & \qquad \qquad \prod_{v\in V_{nr}(T)} 
        \int_{\oM_{0,|H_+(v)|+1}} \Psi(v) \sum_{\ell\in\mathcal{L}(T)}(-1)^{\ell(T)}.
\end{align}
Note that $\int_{\oM_{0,|H_+(v)|+1}} \Psi(v) = a(v)^{|H_+(v)|-2}$ and $\prod_{e\in E(T)}a(e) = \prod_{v\in V_{nr}(T)} a(v)$. Notice also that for any $T$ we have $\sum_{\ell\in\mathcal{L}(T)}(-1)^{\ell(T)} = (-1)^{|E(T)|}$, as shown in the following lemma. Combining these three observations, we conclude.
\end{proof}

\begin{lemma}
    For any rooted tree $T$, we have that $\sum_{\ell\in\mathcal{L}(T)}(-1)^{\ell(T)} = (-1)^{|E(T)|}$.
\end{lemma}

\begin{proof}
For $T$ a tree with a single branch, there is only one $\ell \in \mathcal{L}(T)$, with $\ell(T) = |E(T)|$, therefore the formula holds. For an arbitrary tree $T$, denote by $T'$ the tree obtained by attaching a vertex $v'$ to a vertex $v$ of $T$. 
For $\ell'\in\mathcal{L}(T')$, either $\ell'(v')$ is the same as the level of some vertex of $T$, or $v'$ is the only vertex with level $\ell'(v')$. In the first case, 
$\ell'(T') = \ell(T)$ where $\ell\in\mathcal{L}(T)$ is obtained by restriction of $\ell'$ to $T$; there are exactly $\ell(T)-\ell(v)$ ways of extending $\ell$ to $\ell'$, corresponding to the choices $\ell'(v') = \ell(v)+1, \dots, \ell(T)$. In the second case, the restriction of $\ell'$ to $T$ does not produce a level function, since we need to remove the empty level by subtracting one from the level of each vertex having level larger than the level of $v'$. Denote by $\ell$ the resulting level function on $T$; in this case $\ell(T')=\ell(T)+1$ and there are exactly $\ell(T)-\ell(v)+1$
ways of extending $\ell$ to $\ell'$, corresponding to the possible positions of the new level. It follows that 
\begin{align}
    \sum_{\ell\in\mathcal{L}(T')}(-1)^{\ell(T')} = 
    \sum_{\ell\in\mathcal{L}(T)} (\ell(T)-\ell(v)) (-1)^{\ell(T)} +
    \sum_{\ell\in\mathcal{L}(T)} (\ell(T)-\ell(v)+1) (-1)^{\ell(T)+1}
    = - \sum_{\ell\in\mathcal{L}(T)}(-1)^{\ell(T)}.
\end{align}
Hence attaching a new vertex changes the sign of $\sum_{\ell\in\mathcal{L}(T)}(-1)^{\ell(T)}$. By induction over the vertices (or edges) we conclude. 
\end{proof}

To proceed, we need the following purely combinatorial statement:
\begin{lemma} \label{L-5-8}
For any $n\geq 2$ we have:
\begin{align} \label{eq:comb-identity}
    -\sum_{T\in \SRT_{0,n,1}} (-1)^{|E(T)|} \prod_{v\in V(T)} a(v)^{|H_+(v)|-1} = 
    \Big(-\sum_{i=1}^n a_i\Big)^{n-1}.
\end{align}
\end{lemma}

\begin{proof} 
We proceed by induction over the number of regular legs $n$. For $n=2$ the statement is obviously true, since only the graph with a single vertex contributes to the sum. Observe that, by distinguishing the root vertex and summing over the branches attached to it, the sum over graphs can be splitted as 
\begin{equation}
\sum_{T\in\SRT_{0,n,1}} = \sum_{k=2}^n \frac1{k!} \sum_{\substack{I_1\sqcup \cdots \sqcup I_k = \{1,\dots,n\} \\ I_i\not=\emptyset,\ i=1,\dots,k}} \prod_{i=1}^k \sum_{T\in\SRT_{0,|I_i|,1}}
\end{equation}
where the $k!$ compensates for the labeling of the branches. 
By inductive hypothesis, assume the statement is true for any number of regular legs strictly smaller than $n$ and apply it to every branch in the sum. The the proof of~\eqref{eq:comb-identity} then reduces to proving
\begin{align} \label{eq:comb-identity1}
    -\sum_{k=2}^n \frac{1}{k!} \sum_{\substack{I_1\sqcup \cdots \sqcup I_k = \{1,\dots,n\} \\ I_i\not=\emptyset,\ i=1,\dots,k}}
    \Big(\sum_{i=1}^n a_i\Big)^{k-1} \prod_{i=1}^k (-a_{I_i})^{|I_i|-1} = 
    \Big(-\sum_{i=1}^n a_i\Big)^{n-1},
\end{align}
where $a_I \coloneqq \sum_{i\in I} a_i$, which can be further rewritten as 
\begin{align} \label{eq:comb-identity2}
    \sum_{k=1}^n \frac{1}{k!} \sum_{\substack{I_1\sqcup \cdots \sqcup I_k = \{1,\dots,n\} \\ I_i\not=\emptyset,\ i=1,\dots,k}}
    \Big(\sum_{i=1}^n a_i\Big)^{k-1} \prod_{i=1}^k (-a_{I_i})^{|I_i|-1} = 0.
\end{align}
We can prove this identity by analyzing its restrictions to the divisors $a_i=0$, $i=1,\dots,n$, cf.~\cite[Proof of Eq.~(4.31)]{dekinga2026resonancetransformations22p1minimal}. One can check directly that~\eqref{eq:comb-identity2} holds at $a_n=0$. Indeed, if $n\geq 2$, then
\begin{align} 
    &
    \left(\sum_{k=1}^n \frac{1}{k!} \sum_{\substack{I_1\sqcup \cdots \sqcup I_k = \{1,\dots,n\} \\ I_i\not=\emptyset,\ i=1,\dots,k}}
    \Big(\sum_{i=1}^n a_i\Big)^{k-1} \prod_{i=1}^k (-a_{I_i})^{|I_i|-1}\right)\bigg|_{a_n=0}
    \\ \notag & = \sum_{k=1}^{n-1} \frac{1}{k!} \sum_{\substack{I_1\sqcup \cdots \sqcup I_k = \{1,\dots,n-1\} \\ I_i\not=\emptyset,\ i=1,\dots,k}}
    \Big(\sum_{i=1}^{n-1} a_i\Big)^{k-1} \prod_{i=1}^k (-a_{I_i})^{|I_i|-1} \sum_{j=1}^k (-a_{I_j})
    \\ \notag & \quad + \sum_{k=1}^{n-1} \frac{1}{k!} \sum_{\substack{I_1\sqcup \cdots \sqcup I_k = \{1,\dots,n-1\} \\ I_i\not=\emptyset,\ i=1,\dots,k}}
    \Big(\sum_{i=1}^{n-1} a_i\Big)^{k} \prod_{i=1}^k (-a_{I_i})^{|I_i|-1}
    \\ \notag & = 0,
\end{align}
where the first summand corresponds to the partitions $I_1\sqcup \cdots \sqcup I_k$ of $\{1,\dots,n\}$ such that $n\in I_j$ and $I_j\setminus \{n\}$ is non-empty, and the second summand corresponds to the partitions where one of $I_j$'s is equal to $\{n\}$.
By symmetry in the arguments $a_1,\dots,a_{n}$ we conclude that~\eqref{eq:comb-identity2} also holds at $a_i=0$ for all $i=1,\dots,n-1$. Since the only polynomial of degree $n-1$ in the variables $a_1,\dots,a_n$ that vanishes at $a_i=0$ for all $i=1,\dots,n$ is constant zero, we conclude that~\eqref{eq:comb-identity2} holds, for any $n\geq 2$. 
\end{proof}

Combining the results of Lemmata~\ref{L-5-6} and~\ref{L-5-8}, we arrive at the following expression:
\begin{lemma} \label{lem:B-to-Psi} For any $h\geq 2$ and $n\leq h$ we have:
\begin{align}
    \int_{\oM_{h,n}} \ch_{2h-1} B^0_{h,n} & = \sum_{k=1}^n \frac{1}{k!} \sum_{\substack{I_1\sqcup \cdots \sqcup I_k = \{1,\dots,n\} \\ I_i\not=\emptyset,\ i=1,\dots,k}} \sum_{\substack{d_1,\dots,d_k\geq 0 \\ d_1+\cdots+d_k = h+k-2}} \int_{\oM_{h,k}} \ch_{2h-1} \prod_{i=1}^k \psi_i^{d_i} \prod_{i=1}^k (-1)^{|I_i|-1} a_{I_i}^{|I_i|-1+d_i}
    \label{eq: interesting integrals}
\end{align}
where $a_I\coloneqq \sum_{i\in I} a_i$ for any $I\subseteq \{1,\dots,n\}$, and $I_1\sqcup \cdots \sqcup I_k = \{1,\dots,n\}$ denotes an ordered set partition of $\{1,\dots,n\}$.
\end{lemma}

\begin{proof} Consider Eq.~\eqref{eq:reduction-to-simple-trees}. If we remove the root vertex, we obtain a forest of trees $T_1,\dots,T_k$, $k\geq 1$, that implies a partition of the labels of legs $\{1,\dots,n\}$ into $k$ non-empty subsets $I_1,\dots,I_k$. We have $T_i\in\SRT_{0,|I_i|,1}$, $i=1,\dots,k$. Thus, we can rewrite Eq.~\eqref{eq:reduction-to-simple-trees} as 
    \begin{align} \label{eq:reduction-to-simple-trees-2}
       \int_{\oM_{h,n}} \ch_{2h-1} B^0_{h,n} = & \sum_{k=1}^n \frac{1}{k!} \sum_{\substack{I_1\sqcup \cdots \sqcup I_k = \{1,\dots,n\} \\ I_i\not=\emptyset,\ i=1,\dots,k}} \sum_{\substack{d_1,\dots,d_k\geq 0 \\ d_1+\cdots+d_k = h+k-2}} \int_{\oM_{h,k}} \ch_{2h-1} \prod_{i=1}^k \psi_i^{d_i} a_{I_i}^{d_i} \times
    \\ \notag & \quad 
    \prod_{\substack{i=1\\ |I_i|\geq 2}}^k \bigg(-\sum_{T\in \SRT_{0,|I_i|,1}} (-1)^{|E(T)|} \prod_{v\in V(T)} \tilde a(v)^{|H_+(v)|-1}\bigg),
    \end{align}
    where by $\tilde a(v)$ we mean that for a tree in $\SRT_{0,|I_i|,1}$ in the $i$-th factor the parameters associated to legs are taken in $\{a_i\}_{i\in I_i}$. Note that for convenience we used an ordered set partition $(I_1,\dots,I_k)$ of $\{1,\dots,n\}$ and added a prefactor $\tfrac{1}{k!}$. Applying Eq.~\eqref{eq:comb-identity} to the factors in the second line, we obtain the statement of the lemma.
\end{proof}

The intersection numbers $\int_{\overline{\mathcal{M}}_{h,n}}{\rm ch}{}_{2h-1}B^{0}_{h,n}$ satisfy the following dilaton-type property.
\begin{lemma} \label{lem:Quasi-Dilaton} 
For $2\leq n\leq h$ we have
\begin{equation}
{\rm Coeff}_{\left(a_{n}\right)^{1}}\int_{\overline{\mathcal{M}}_{h,n}}{\rm ch}{}_{2h-1}B^{0}_{h,n}=(h-n+1)\int_{\overline{\mathcal{M}}_{h,n-1}}{\rm ch}{}_{2h-1}B^{0}_{h,n-1},
\end{equation}
where ${\rm Coeff}_{\left(a_{n}\right)^{1}}$ denotes the operator extracting the coefficient of $\left(a_{n}\right)^{1}$ while treating the remaining variables $a_{1},\dots,a_{n-1}$ as formal parameters.
\end{lemma}

\begin{proof}
First notice that, since $n\leq h$, the integrals on both sides of the equality can be represented using Lemma~\ref{lem:B-to-Psi}. We start from the left hand side and extract the coefficient of $\left(a_{n}\right)^{1}$:
\begin{align}
\frac{\partial}{\partial a_{n}}\int_{\overline{\mathcal{M}}_{h,n}}{\rm ch}{}_{2h-1}B^{0}_{h,n}\Big|_{a_{n}=0}
\end{align}
using (\ref{eq: interesting integrals}). The derivative $\frac{\partial}{\partial a_{n}}a^{\left|I_{j}\right|-1+d_{j}}_{I_{j}}\big|_{a_{n}=0}$ is nonzero only if $n\in I_{j},$  we get
\begin{align}
 \frac{\partial}{\partial a_{n}}\int_{\overline{\mathcal{M}}_{g,n}}B^{0}_{g,n}{\rm ch}{}_{2g-1}\Big|_{a_{n}=0}  =\sum^{n}_{k=1} \frac{1}{k!}\sum^{k}_{j=1}\sum_{\substack{I_{1}\sqcup\cdots\sqcup I_{k}=\{1,\dots,n\}\\
I_{i}\not=\emptyset,\ i=1,\dots,k\\
n\in I_{j}
}
}\sum_{\substack{d_{1},\dots,d_{k}\geq0\\
d_{1}+\cdots+d_{k}=h+k-2\\
|I_{j}|+d_{j}\geq2
}
}-\left(\left|I_{j}\right|-1+d_{j}\right)\label{eq: derivative integral} \\
 \qquad \times \int_{\overline{\mathcal{M}}_{h,k}}{\rm ch}_{2h-1}\prod^{k}_{i=1}\psi^{d_{i}}_{i}(-1)^{|I_{j}|-2} a^{\left|I_{j}\right|-2+d_{j}}_{I_{j}}\big|_{a_{n}=0}\prod^{k}_{\substack{i=1\\
i\neq j
}
}(-1)^{|I_{i}|-1}a^{|I_{i}|-1+d_{i}}_{I_{i}}.\nonumber 
\end{align}
There are two kind of terms. 
\begin{itemize}
\item If $\left|I_{j}\right|>1$, we denote $\tilde{I}_{j}\coloneqq I_{j}\backslash\left\{ n\right\} $ and make the change of summation given by the bijection of ordered set partition of $\left\{ 1,\dots,n\right\} $ into $k$ blocks and ordered set partition of $\left\{ 1,\dots,n-1\right\} $ into $k$ blocks: 
\begin{align}
\left(I_{1},\dots,I_{k}\right)\rightarrow\left(I_{1},\dots,\tilde{I}_{j},\dots,I_{k}\right).
\end{align}
Since $a_{I_{j}}|_{a_{n}=0}=a_{\tilde{I}_{j}}$, $\left|I_{j}\right|-2+d_{j}=|\tilde{I}_{j}|-1+d_{j}$ and $(-1)^{|I_{j}|-2}=(-1)^{|\tilde{I}_{j}|-1}$, the factor of $-\left(\left|\tilde{I}_{j}\right|+d_{j}\right)$ in (\ref{eq: derivative integral}) is independent of $j$, so the sum over $j$ yields 
\begin{align}
\sum^{k}_{j=1}-\bigl(|\tilde{I}_{j}|+d_{j}\bigr)=-\bigl(n-3+h+k\bigr).
\end{align}
\item If $\left|I_{j}\right|=1$, then the only non-vanishing contribution is for $d_{j}=1$, and thus $\left|I_{j}\right|-1+d_{j}=1$. In that situation, we make the change of summation given by the bijection of ordered set partition of $\left\{ 1,\dots,n\right\} $ into $k$ blocks and ordered set partition $\left\{ 1,\dots,n-1\right\} $ into $\left(k-1\right)$ blocks: 
\begin{align}
\left(I_{1},\dots,I_{k}\right)\rightarrow\left(I_{1},\dots,\widehat{I_{j}},\dots,I_{k}\right),
\end{align}
where $\widehat{I_{j}}$ indicates that $I_{j}$ is removed from the tuple. In addition, we apply the dilaton equation on the root 
\begin{align}
\int_{\overline{\mathcal{M}}_{h,k}}{\rm ch}_{2h-1}\psi^{d_{1}}_{1}\cdots\psi_{j}\cdots\psi^{d_{k}}_{k}=\left(2h-3+k\right)\int_{\overline{\mathcal{M}}_{h,k-1}}{\rm ch}_{2h-1}\psi^{d_{1}}_{1}\cdots\widehat{\psi_{j}}\cdots\psi^{d_{k}}_{k}.
\end{align}
The resulting summand in (\ref{eq: derivative integral}) is again independent of $j$, summing over $j=1,\dots,k$ therefore produces a multiplicity $k$, which turns the prefactor $\frac{1}{k!}$ into $\frac{1}{(k-1)!}$. Setting $k'\coloneqq k-1$, the constraint becomes $d_{1}+\cdots+d_{k'}=h+k'-2$ and the dilaton factor reads $2h-2+k'$.
\end{itemize}
Gathering both types of terms in (\ref{eq: derivative integral}), noticing that terms of the first kind force $k\leq n-1$ and terms of the second kind force $k\geq2$, we get 

\begin{align}
& \frac{\partial}{\partial a_{n}}\int_{\overline{\mathcal{M}}_{g,n}}B^{0}_{g,n}{\rm ch}{}_{2g-1}\Big|_{a_{n}=0}=\sum^{n-1}_{k=1} \frac{1}{k!}\left[-\left(n-3+h+k\right)+\left(2h-2+k\right)\right] \\
& \qquad \qquad  \times  \sum_{\substack{I_{1}\sqcup\cdots\sqcup I_{k}=\{1,\dots,n-1\}\\
I_{i}\not=\emptyset,\ i=1,\dots,k
}
} \sum_{\substack{d_{1},\dots,d_{k}\geq0\\
d_{1}+\cdots+d_{k}=h+k-2
}
}  \int_{\overline{\mathcal{M}}_{h,k}} {\rm ch}_{2h-1}\prod^{k}_{i=1}\psi^{d_{i}}_{i}\prod^{k}_{\substack{i=1}
}(-1)^{|I_{i}|-1}a^{|I_{i}|-1+d_{i}}_{I_{i}},\nonumber 
\end{align}
which is exactly $(h-n+1)\int_{\overline{\mathcal{M}}_{h,n-1}}{\rm ch}{}_{2h-1}B^{0}_{h,n-1}$ by Lemma~\ref{lem:B-to-Psi}.
\end{proof}

\subsubsection{Proof of Lemma~\ref{lemma:non-vanishing}} \label{subsubsec:ProofLemmaNonvanishing}

Our main tools are Lemmata~\ref{lem:Quasi-Dilaton},~\ref{lem:B-to-Psi}, and the following formula, called Faber's socle intersection number formula: 

\begin{lemma}\label{lem:Faber} Let $g\geq 2$, $k\geq 1$, and $d_1,\dots,d_k\geq 0$ such that $d_1+\cdots+d_k=g-2+k$. 
    If at most one of the $d_{1},\dots,d_{k}$ is vanishing, the intersection number $\int_{\overline{\mathcal{M}}_{g,k}}{\rm ch}_{2g-1}\psi^{d_{1}}_{1}\cdots\psi^{d_{k}}_{k}$ can be evaluated as
\begin{align}
\int_{\overline{\mathcal{M}}_{g,k}}{\ch}_{2g-1}\psi^{d_{1}}_{1}\cdots\psi^{d_{k}}_{k}=\frac{\left(-1\right)^{g-1}\mathrm{B}_{2g}\left(2g-3+k\right)!}{2^{2g-1}\cdot\left(2g\right)!}\prod^{k}_{i=1}\frac{1}{\left(2d_{i}-1\right)!!},
\end{align}
where $\mathrm{B}_{2g}$ are the Bernoulli numbers.
\end{lemma}
It was originally conjectured in~\cite{Faber}, first proved in~\cite{GetPan} by reduction to Virasoro constraints subsequently proved in~\cite{Givental}, and by now there are many different proofs, see~\cite{BlotShaSingh} for the most recent one and a short overview of other existing proofs. 

Now, we have:
\begin{align}
    & \int_{\oM_{h,h}} \ch_{2h-1} \Coeff{a_1\cdots a_{h-1} a_h^{h-1}} B^0_{h,h}  \overset{\text{Lem.~\ref{lem:Quasi-Dilaton}}}{=} (h-1)! \int_{\oM_{h,1}} \ch_{2h-1} \Coeff{a_1^{h-1}} B^0_{h,1} 
    \\ \notag & \overset{\text{Lem.~\ref{lem:B-to-Psi}}}{=} (h-1)! \int_{\oM_{h,1}} \ch_{2h-1} \psi_1^{h-1} \overset{\text{Lem.~\ref{lem:Faber}}}{=} (h-1)!\frac{\left(-1\right)^{h-1}\mathrm{B}_{2h}\left(2h-2\right)!}{2^{2h-1}\cdot\left(2h\right)!(2h-3)!!}.
\end{align}
This implies \eqref{eq:Odd-h-nonvanishing} (after further simplification for odd $h$). 

Furthermore,
\begin{align}
    & \int_{\oM_{h,h}} \ch_{2h-1} \Coeff{a_1\cdots a_{h-2} a_{h-1}^2 a_h^{h-2}} B^0_{h,h} \overset{\text{Lem.~\ref{lem:Quasi-Dilaton}}}{=} (h-2)! \int_{\oM_{h,2}} \ch_{2h-1} \Coeff{a_{1}^2 a_2^{h-2}} B^0_{h,2} 
    \\ \notag &
    \overset{\text{Lem.~\ref{lem:B-to-Psi}}}{=}  
    (h-2)! \int_{\oM_{h,2}} \ch_{2h-1} \psi_1^2\psi_2^{h-2} -
    (h-2)!\Coeff{a_{1}^2 a_2^{h-2}} \int_{\oM_{h,1}} \ch_{2h-1} \psi_1^{h-1} (a_1+a_2)^{h}
    \\ \notag 
    & \overset{\text{Lem.~\ref{lem:Faber}}}{=} (h-2)! \frac{\left(-1\right)^{h-1}\mathrm{B}_{2h}}{2^{2h-1}\cdot\left(2h\right)!} \bigg( \frac{(2h-1)!}{3!! (2h-5)!!} - \frac{h!}{2!(h-2)!} \frac{(2h-2)!}{(2h-3)!!}\bigg) 
    \\ \notag & = 
    (h-2)! \frac{\left(-1\right)^{h-1}\mathrm{B}_{2h}}{2^{2h-1}\cdot\left(2h\right)!}\cdot \frac{(2h-2)!(5h^2-13h+6)}{6(2h-3)!!}.
\end{align}
Substituting the results of these computations in~\eqref{eq:Even-h-nonvanishing} for even $h$, we have:
\begin{align}
        & \frac{2-h}{(h-2)!}  (h-1)!\frac{-\mathrm{B}_{2h}\left(2h-2\right)!}{2^{2h-1}\cdot\left(2h\right)!(2h-3)!!}
    + \frac{2}{(h-2)!} (h-2)! \frac{-\mathrm{B}_{2h}}{2^{2h-1}\cdot\left(2h\right)!}\cdot \frac{(2h-2)!(5h^2-13h+6)}{6(2h-3)!!} 
    \\ \notag & = \frac{-\mathrm{B}_{2h}\left(2h-2\right)!}{2^{2h-1}\cdot\left(2h\right)!(2h-3)!!}\cdot  \bigg((2-h)(h-1) + \frac 13 (5h^2-13h+6)\bigg) = \frac{-\mathrm{B}_{2h}\left(2h-2\right)!}{2^{2h-1}\cdot\left(2h\right)!(2h-3)!!}\cdot \frac{2h(h-2)}{3}, 
\end{align}
which implies~\eqref{eq:Even-h-nonvanishing}.

\subsection{Proof of Theorem~\ref{thm:main}} 

We fix arbitrary parameters $e_1$ and $e_3$ and prove that if $\mathrm{P}$ is constant then all further $e_{2h-1}$ are equal to zero. 
Assume it is not the case, so $\mathrm{P}$ is constant but not all higher $e_{2h-1}$ are equal to zero. Let $h\geq 3$ be the minimal possible number such that $e_{2h-1}\not=0$, besides $e_1$ and $e_3$.

If $h$ is odd, we consider the coefficient of $\epsilon^{2h} w_1^{h-1} \partial_x^{h+2}$ in $\mathrm{P}$. It is a homogeneous polynomial in $e_{2i-1}$, $i \geq 1$, of homogeneous degree $2h-1$ with respect to the natural grading of the ring of symmetric functions. Since $e_{2h-1}$ is the minimal non-vanishing elementary symmetric function besides $e_1$ and $e_3$, the coefficient of $\epsilon^{2h} w_1^{h-1} \partial_x^{h+2}$ is given by 
\begin{align}
    (2h-1)!e_{2h-1}\cdot \frac{1+(-1)^{h-1}}{(h-1)!} \int_{\oM_{h,h}} \ch_{2h-1} \Coeff{a_1\cdots a_{h-1} a_h^{h-1}} B^0_{h,h} + Q(e_1,e_3),
\end{align}
where $Q(e_1,e_3)$ is an unknown polynomial. By Prop.~\ref{prop:FVH}, $\mathrm{P}|_{e_{2k-1}=0,\, k\geq 3}$ is constant in $w_i$, $i\geq 0$. Hence, $Q(e_1,e_3)=0$.
Hence, if $\mathrm{P}$ is constant in $w$, the coefficient of $\epsilon^{2h} w_1^{h-1} \partial_x^{h+2}$, which is reduced to
\begin{align}
    (2h-1)!e_{2h-1}\cdot \frac{1+(-1)^{h-1}}{(h-1)!} \int_{\oM_{h,h}} \ch_{2h-1} \Coeff{a_1\cdots a_{h-1} a_h^{h-1}} B^0_{h,h},
\end{align}
must be equal to zero. Using the non-vanishing property established in Lemma~\ref{lemma:non-vanishing}, we conclude that the constancy of $\mathrm{P}$ implies that $e_{2h-1}$ vanishes. 

If $h$ is even, we repeat the argument above \emph{mutatis mutandis}, considering the coefficient of $\epsilon^{2h} w_1^{h-2}w_2 \partial_x^{h+1}$ in $\mathrm{P}$.

\printbibliography

@misc{BLS-DRDZ,
      title={On the strong DR/DZ equivalence conjecture}, 
      author={Xavier Blot and Danilo Lewa{\'n}ski and Sergey Shadrin},
      year={2024},
      eprint={2405.12334},
      archivePrefix={arXiv},
      primaryClass={math.AG},
      url={https://arxiv.org/abs/2405.12334}, 
}

@misc{blot2026descendantsintegrableobservablescohomological,
      title={Beyond descendants: integrable observables for cohomological field theories}, 
      author={Xavier Blot and Danilo Lewański and Sergey Shadrin},
      year={2026},
      eprint={2605.22236},
      archivePrefix={arXiv},
      primaryClass={math.AG},
      url={https://arxiv.org/abs/2605.22236}, 
}

@misc{blot2024rootedtreeslevelstructures,
      title={Rooted trees with level structures, $\Omega$-classes and double ramification cycles}, 
      author={Xavier Blot and Danilo Lewański and Sergey Shadrin},
      year={2024},
      eprint={2406.06205},
      archivePrefix={arXiv},
      primaryClass={math.AG},
      url={https://arxiv.org/abs/2406.06205}, 
}

@article {BLRS,
    AUTHOR = {Blot, Xavier and Lewa\'{n}ski, Danilo and Rossi, Paolo and
              Shadrin, Sergei},
     TITLE = {Stable tree expressions with omega-classes and double
              ramification cycles},
   JOURNAL = {J. Geom. Phys.},
  FJOURNAL = {Journal of Geometry and Physics},
    VOLUME = {209},
      YEAR = {2025},
     PAGES = {Paper No. 105391, 17},
      ISSN = {0393-0440,1879-1662},
   MRCLASS = {14H10},
  MRNUMBER = {4839088},
MRREVIEWER = {Zhengning\ Hu},
       DOI = {10.1016/j.geomphys.2024.105391},
       URL = {https://doi.org/10.1016/j.geomphys.2024.105391},
}

@article {BSS25,
    AUTHOR = {Blot, Xavier and Sauvaget, Adrien and Shadrin, Sergey},
	TITLE = {The master relation for polynomiality and equivalences of
	integrable systems},
	JOURNAL = {Bull. Lond. Math. Soc.},
	FJOURNAL = {Bulletin of the London Mathematical Society},
	VOLUME = {57},
	YEAR = {2025},
	NUMBER = {2},
	PAGES = {599--604},
	ISSN = {0024-6093,1469-2120},
	MRCLASS = {14H10 (37K10 53D45)},
	MRNUMBER = {4861898},
	MRREVIEWER = {Alessandro\ Giacchetto},
	DOI = {10.1112/blms.13215},
	URL = {https://doi.org/10.1112/blms.13215},
}

@article {Buryak-DR,
    AUTHOR = {Buryak, A.},
     TITLE = {Double ramification cycles and integrable hierarchies},
   JOURNAL = {Comm. Math. Phys.},
  FJOURNAL = {Communications in Mathematical Physics},
    VOLUME = {336},
      YEAR = {2015},
    NUMBER = {3},
     PAGES = {1085--1107},
      ISSN = {0010-3616,1432-0916},
   MRCLASS = {37K20 (14H10 35Q51 37K25 53D45)},
  MRNUMBER = {3324138},
MRREVIEWER = {Pierre\ A.\ Lochak},
       DOI = {10.1007/s00220-014-2235-2},
       URL = {https://doi.org/10.1007/s00220-014-2235-2},
}

@article{BDGR1,
 author = {Buryak, A. and Dubrovin, B. and Gu{\'e}r{\'e}, J. and Rossi, P.},
 title = {Tau-structure for the double ramification hierarchies},
 fjournal = {Communications in Mathematical Physics},
 journal = {Comm. Math. Phys.},
 issn = {0010-3616},
 volume = {363},
 number = {1},
 pages = {191--260},
 year = {2018},
 language = {English},
 doi = {10.1007/s00220-018-3235-4},
 zbMATH = {6967377},
 Zbl = {1431.53095}
}

@article{BDGR20,
 author = {Buryak, A. and Dubrovin, B. and Gu{\'e}r{\'e}, J. and Rossi, P.},
 title = {Integrable systems of double ramification type},
 fjournal = {IMRN. International Mathematics Research Notices},
 journal = {Int. Math. Res. Not.},
 issn = {1073-7928},
 volume = {2020},
 number = {24},
 pages = {10381--10446},
 year = {2020},
 language = {English},
 doi = {10.1093/imrn/rnz029},
 zbMATH = {7323449},
 Zbl = {1464.37071}
}

@article{BPS1,
 author = {Buryak, A. and Posthuma, H. and Shadrin, S.},
 title = {A polynomial bracket for the {Dubrovin}-{Zhang} hierarchies},
 fjournal = {Journal of Differential Geometry},
 journal = {J. Differ. Geom.},
 issn = {0022-040X},
 volume = {92},
 number = {1},
 pages = {153--185},
 year = {2012},
 language = {English},
 doi = {10.4310/jdg/1352211225},
 zbMATH = {6130615},
 Zbl = {1259.53079}
}

@article {BPS-2,
    AUTHOR = {Buryak, A. and Posthuma, H. and Shadrin, S.},
     TITLE = {On deformations of quasi-{M}iura transformations and the
              {D}ubrovin-{Z}hang bracket},
   JOURNAL = {J. Geom. Phys.},
  FJOURNAL = {Journal of Geometry and Physics},
    VOLUME = {62},
      YEAR = {2012},
    NUMBER = {7},
     PAGES = {1639--1651},
      ISSN = {0393-0440,1879-1662},
   MRCLASS = {37K10 (37K05 37K35)},
  MRNUMBER = {2922026},
MRREVIEWER = {Hao\ Xu},
       DOI = {10.1016/j.geomphys.2012.03.006},
       URL = {https://doi.org/10.1016/j.geomphys.2012.03.006},
}

@article {Buryak-Rossi-recursion,
    AUTHOR = {Buryak, Alexandr and Rossi, Paolo},
     TITLE = {Recursion relations for double ramification hierarchies},
   JOURNAL = {Comm. Math. Phys.},
  FJOURNAL = {Communications in Mathematical Physics},
    VOLUME = {342},
      YEAR = {2016},
    NUMBER = {2},
     PAGES = {533--568},
      ISSN = {0010-3616,1432-0916},
   MRCLASS = {37K10 (37K65 81R12)},
  MRNUMBER = {3459159},
MRREVIEWER = {Oktay\ K.\ Pashaev},
       DOI = {10.1007/s00220-015-2535-1},
       URL = {https://doi.org/10.1007/s00220-015-2535-1},
}

@article{BS22,
 author = {Buryak, A. and Shadrin, S.},
 title = {Tautological relations and integrable systems},
 fjournal = {{\'E}pijournal de G{\'e}om{\'e}trie Alg{\'e}brique. EPIGA},
 journal = {{\'E}pijournal de G{\'e}om. Alg{\'e}br., EPIGA},
 issn = {2491-6765},
 volume = {8},
 pages = {44},
 note = {Id/No 12},
 year = {2024},
 language = {English},
 doi = {10.46298/epiga.2024.10382},
 zbMATH = {7939129},
 Zbl = {1552.14017}
}

@misc{dubrovin2001normalformshierarchiesintegrable,
      title={Normal forms of hierarchies of integrable {PDE}s, {F}robenius manifolds and {G}romov - {W}itten invariants}, 
      author={Dubrovin, B. and Zhang, Y.},
      year={2001},
      eprint={math/0108160},
      archivePrefix={arXiv},
      primaryClass={math.DG},
      url={https://arxiv.org/abs/math/0108160}, 
      note={See also the 2005 updated version.},
}

@article {Teleman,
	AUTHOR = {Teleman, Constantin},
	TITLE = {The structure of 2{D} semi-simple field theories},
	JOURNAL = {Invent. Math.},
	FJOURNAL = {Inventiones Mathematicae},
	VOLUME = {188},
	YEAR = {2012},
	NUMBER = {3},
	PAGES = {525--588},
	ISSN = {0020-9910,1432-1297},
	MRCLASS = {57R56 (18D10 53D45)},
	MRNUMBER = {2917177},
	MRREVIEWER = {Julia\ Bergner},
	DOI = {10.1007/s00222-011-0352-5},
	URL = {https://doi.org/10.1007/s00222-011-0352-5},
}

@article {Hodge-FVH,
	AUTHOR = {Liu, Si-Qi and Yang, Di and Zhang, Youjin and Zhou, Chunhui},
	TITLE = {The {H}odge-{FVH} correspondence},
	JOURNAL = {J. Reine Angew. Math.},
	FJOURNAL = {Journal f\"{u}r die Reine und Angewandte Mathematik. [Crelle's
	Journal]},
	VOLUME = {775},
	YEAR = {2021},
	PAGES = {259--300},
	ISSN = {0075-4102,1435-5345},
	MRCLASS = {14H70 (37K10)},
	MRNUMBER = {4265184},
	MRREVIEWER = {Lihua\ Wu},
	DOI = {10.1515/crelle-2020-0051},
	URL = {https://doi.org/10.1515/crelle-2020-0051},
}

@article {Hodge-GUE,
	AUTHOR = {Dubrovin, Boris and Liu, Si-Qi and Yang, Di and Zhang, Youjin},
	TITLE = {Hodge-{GUE} correspondence and the discrete {K}d{V} equation},
	JOURNAL = {Comm. Math. Phys.},
	FJOURNAL = {Communications in Mathematical Physics},
	VOLUME = {379},
	YEAR = {2020},
	NUMBER = {2},
	PAGES = {461--490},
	ISSN = {0010-3616,1432-0916},
	MRCLASS = {14H10 (14H81 33C60 37K10 81T45)},
	MRNUMBER = {4156215},
	MRREVIEWER = {Xiaoxue\ Xu},
	DOI = {10.1007/s00220-020-03846-6},
	URL = {https://doi.org/10.1007/s00220-020-03846-6},
}

@article {FVH-definition,
	AUTHOR = {Liu, Si-Qi and Zhang, Youjin and Zhou, Chunhui},
	TITLE = {Fractional {V}olterra hierarchy},
	JOURNAL = {Lett. Math. Phys.},
	FJOURNAL = {Letters in Mathematical Physics},
	VOLUME = {108},
	YEAR = {2018},
	NUMBER = {2},
	PAGES = {261--283},
	ISSN = {0377-9017,1573-0530},
	MRCLASS = {37K10 (53D45)},
	MRNUMBER = {3748364},
	MRREVIEWER = {Guoliang\ He},
	DOI = {10.1007/s11005-017-1006-3},
	URL = {https://doi.org/10.1007/s11005-017-1006-3},
}

@article {CubicHodge-Loop,
	AUTHOR = {Liu, Si-Qi and Yang, Di and Zhang, Youjin and Zhou, Chunhui},
	TITLE = {The loop equation for special cubic {H}odge integrals},
	JOURNAL = {J. Differential Geom.},
	FJOURNAL = {Journal of Differential Geometry},
	VOLUME = {121},
	YEAR = {2022},
	NUMBER = {2},
	PAGES = {341--368},
	ISSN = {0022-040X,1945-743X},
	MRCLASS = {14N35 (53D45)},
	MRNUMBER = {4466672},
	MRREVIEWER = {Hsian-Hua\ Tseng},
	DOI = {10.4310/jdg/1659987894},
	URL = {https://doi.org/10.4310/jdg/1659987894},
}

@article {DLYZ-Hodge-conjecture,
	AUTHOR = {Dubrovin, Boris and Liu, Si-Qi and Yang, Di and Zhang, Youjin},
	TITLE = {Hodge integrals and tau-symmetric integrable hierarchies of
	{H}amiltonian evolutionary {PDE}s},
	JOURNAL = {Adv. Math.},
	FJOURNAL = {Advances in Mathematics},
	VOLUME = {293},
	YEAR = {2016},
	PAGES = {382--435},
	ISSN = {0001-8708,1090-2082},
	MRCLASS = {53D45 (37K10)},
	MRNUMBER = {3474326},
	MRREVIEWER = {Zhengyu\ Zong},
	DOI = {10.1016/j.aim.2016.01.018},
	URL = {https://doi.org/10.1016/j.aim.2016.01.018},
}

@article {Buryak-ILW,
	AUTHOR = {Buryak, A.},
	TITLE = {Dubrovin-{Z}hang hierarchy for the {H}odge integrals},
	JOURNAL = {Commun. Number Theory Phys.},
	FJOURNAL = {Communications in Number Theory and Physics},
	VOLUME = {9},
	YEAR = {2015},
	NUMBER = {2},
	PAGES = {239--272},
	ISSN = {1931-4523,1931-4531},
	MRCLASS = {37K05 (14H10)},
	MRNUMBER = {3361294},
	MRREVIEWER = {Ahmed\ Lesfari},
	DOI = {10.4310/CNTP.2015.v9.n2.a1},
	URL = {https://doi.org/10.4310/CNTP.2015.v9.n2.a1},
}

@incollection {Witten-KdV,
	AUTHOR = {Witten, Edward},
	TITLE = {Two-dimensional gravity and intersection theory on moduli
	space},
	BOOKTITLE = {Surveys in differential geometry ({C}ambridge, {MA}, 1990)},
	PAGES = {243--310},
	PUBLISHER = {Lehigh Univ., Bethlehem, PA},
	YEAR = {1991},
	ISBN = {0-8218-0168-6},
	MRCLASS = {32G15 (14C17 14H15 32G81 58F07 81T40)},
	MRNUMBER = {1144529},
	MRREVIEWER = {Steven\ Rosenberg},
    DOI = {10.4310/SDG.1990.v1.n1.a5},
}

@article {Kontsevich-KdV,
	AUTHOR = {Kontsevich, Maxim},
	TITLE = {Intersection theory on the moduli space of curves and the
	matrix {A}iry function},
	JOURNAL = {Comm. Math. Phys.},
	FJOURNAL = {Communications in Mathematical Physics},
	VOLUME = {147},
	YEAR = {1992},
	NUMBER = {1},
	PAGES = {1--23},
	ISSN = {0010-3616,1432-0916},
	MRCLASS = {32G15 (14H15 58F07 81T40)},
	MRNUMBER = {1171758},
	MRREVIEWER = {Claude\ Itzykson},
	URL = {http://projecteuclid.org/euclid.cmp/1104250524},
    DOI = {10.1007/BF02099526},
}

@incollection {Mumford,
    AUTHOR = {Mumford, David},
     TITLE = {Towards an enumerative geometry of the moduli space of curves},
 BOOKTITLE = {Arithmetic and geometry, {V}ol. {II}},
    SERIES = {Progr. Math.},
    VOLUME = {36},
     PAGES = {271--328},
 PUBLISHER = {Birkh\"{a}user Boston, Boston, MA},
      YEAR = {1983},
      ISBN = {3-7643-3133-X},
   MRCLASS = {14H10 (14C15)},
  MRNUMBER = {717614},
MRREVIEWER = {Werner\ Kleinert},
doi = {10.1007/978-1-4757-9286-7_12},
}

@misc{dekinga2026resonancetransformations22p1minimal,
      title={Resonance transformations for the $(2,2p+1)$ minimal string via $x-y$ swap: a proof of Artemev's conjecture}, 
      author={Kornelis Dekinga and Sergey Shadrin and Erik Verlinde},
      year={2026},
      eprint={2606.04854},
      archivePrefix={arXiv},
      primaryClass={hep-th},
      url={https://arxiv.org/abs/2606.04854}, 
}

@article {BuryakGuereRossi,
    AUTHOR = {Buryak, Alexandr and Gu\'{e}r\'{e}, J\'{e}r\'{e}my and Rossi,
              Paolo},
     TITLE = {D{R}/{DZ} equivalence conjecture and tautological relations},
   JOURNAL = {Geom. Topol.},
  FJOURNAL = {Geometry \& Topology},
    VOLUME = {23},
      YEAR = {2019},
    NUMBER = {7},
     PAGES = {3537--3600},
      ISSN = {1465-3060,1364-0380},
   MRCLASS = {14H10 (37K10)},
  MRNUMBER = {4059088},
MRREVIEWER = {\'{A}lvaro\ Ant\'{o}n Sancho},
       DOI = {10.2140/gt.2019.23.3537},
       URL = {https://doi.org/10.2140/gt.2019.23.3537},
}

@article {FaberPanda,
    AUTHOR = {Faber, C. and Pandharipande, R.},
     TITLE = {Hodge integrals and {G}romov-{W}itten theory},
   JOURNAL = {Invent. Math.},
  FJOURNAL = {Inventiones Mathematicae},
    VOLUME = {139},
      YEAR = {2000},
    NUMBER = {1},
     PAGES = {173--199},
      ISSN = {0020-9910,1432-1297},
   MRCLASS = {14N35},
  MRNUMBER = {1728879},
MRREVIEWER = {Jim\ A.\ Bryan},
       DOI = {10.1007/s002229900028},
       URL = {https://doi.org/10.1007/s002229900028},
}

@incollection {Faber,
    AUTHOR = {Faber, Carel},
     TITLE = {A conjectural description of the tautological ring of the
              moduli space of curves},
 BOOKTITLE = {Moduli of curves and abelian varieties},
    SERIES = {Aspects Math., E33},
     PAGES = {109--129},
 PUBLISHER = {Friedr. Vieweg, Braunschweig},
      YEAR = {1999},
      ISBN = {3-528-03125-5},
   MRCLASS = {14H10 (14C15 14C17 14N35)},
  MRNUMBER = {1722541},
MRREVIEWER = {Elham\ Izadi},
doi = {10.1007/978-3-322-90172-9_6},
}

@article {GetPan,
    AUTHOR = {Getzler, E. and Pandharipande, R.},
     TITLE = {Virasoro constraints and the {C}hern classes of the {H}odge
              bundle},
   JOURNAL = {Nuclear Phys. B},
  FJOURNAL = {Nuclear Physics. B. Theoretical, Phenomenological, and
              Experimental High Energy Physics. Quantum Field Theory and
              Statistical Systems},
    VOLUME = {530},
      YEAR = {1998},
    NUMBER = {3},
     PAGES = {701--714},
      ISSN = {0550-3213,1873-1562},
   MRCLASS = {14N35},
  MRNUMBER = {1653492},
MRREVIEWER = {Alexandre\ I.\ Kabanov},
       DOI = {10.1016/S0550-3213(98)00517-3},
       URL = {https://doi.org/10.1016/S0550-3213(98)00517-3},
}

@incollection {Givental,
    AUTHOR = {Givental, Alexander B.},
     TITLE = {Gromov-{W}itten invariants and quantization of quadratic
              {H}amiltonians},
      NOTE = {Dedicated to the memory of I. G. Petrovskii on the occasion of
              his 100th anniversary},
   JOURNAL = {Mosc. Math. J.},
  FJOURNAL = {Moscow Mathematical Journal},
    VOLUME = {1},
      YEAR = {2001},
    NUMBER = {4},
     PAGES = {551--568, 645},
      ISSN = {1609-3321,1609-4514},
   MRCLASS = {53D45 (14N35)},
  MRNUMBER = {1901075},
MRREVIEWER = {Domenico\ Fiorenza},
       DOI = {10.17323/1609-4514-2001-1-4-551-568},
       URL = {https://doi.org/10.17323/1609-4514-2001-1-4-551-568},
}

@article {BlotShaSingh,
    AUTHOR = {Blot, Xavier and Shadrin, Sergey and Singh, Ishan Jaztar},
     TITLE = {Faber's socle intersection numbers via {G}romov-{W}itten
              theory of elliptic curve},
   JOURNAL = {Bull. Lond. Math. Soc.},
  FJOURNAL = {Bulletin of the London Mathematical Society},
    VOLUME = {57},
      YEAR = {2025},
    NUMBER = {9},
     PAGES = {2698--2707},
      ISSN = {0024-6093,1469-2120},
   MRCLASS = {14H10 (14H52 14N35)},
  MRNUMBER = {4956739},
MRREVIEWER = {Tatsunari\ Watanabe},
       DOI = {10.1112/blms.70117},
       URL = {https://doi.org/10.1112/blms.70117},
}

@article {Alexandrov,
    AUTHOR = {Alexandrov, Alexander},
     TITLE = {K{P} integrability of triple {H}odge integrals. {I}. {F}rom
              {G}ivental group to hierarchy symmetries},
   JOURNAL = {Commun. Number Theory Phys.},
  FJOURNAL = {Communications in Number Theory and Physics},
    VOLUME = {15},
      YEAR = {2021},
    NUMBER = {3},
     PAGES = {615--650},
      ISSN = {1931-4523,1931-4531},
   MRCLASS = {37K10 (14N10 14N35 81R10)},
  MRNUMBER = {4290562},
MRREVIEWER = {Hsian-Hua\ Tseng},
       DOI = {10.4310/CNTP.2021.v15.n3.a6},
       URL = {https://doi.org/10.4310/CNTP.2021.v15.n3.a6},
}

@incollection {Takasaki,
    AUTHOR = {Takasaki, Kanehisa},
     TITLE = {Cubic {H}odge integrals and integrable hierarchies of
              {V}olterra type},
 BOOKTITLE = {Integrability, quantization, and geometry. {I}. {I}ntegrable
              systems},
    SERIES = {Proc. Sympos. Pure Math.},
    VOLUME = {103},
     PAGES = {481--502},
 PUBLISHER = {Amer. Math. Soc., Providence, RI},
      YEAR = {[2021] \copyright 2021},
      ISBN = {978-1-4704-5591-0},
   MRCLASS = {37K10 (14N35 53D45)},
  MRNUMBER = {4285689},
MRREVIEWER = {Guoliang\ He},
       DOI = {10.1090/pspum/103.1/01844},
       URL = {https://doi.org/10.1090/pspum/103.1/01844},
}

@article {LiuLiuZhou,
    AUTHOR = {Liu, Chiu-Chu Melissa and Liu, Kefeng and Zhou, Jian},
     TITLE = {A proof of a conjecture of {M}ari\~{n}o-{V}afa on {H}odge
              integrals},
   JOURNAL = {J. Differential Geom.},
  FJOURNAL = {Journal of Differential Geometry},
    VOLUME = {65},
      YEAR = {2003},
    NUMBER = {2},
     PAGES = {289--340},
      ISSN = {0022-040X,1945-743X},
   MRCLASS = {14N35 (14H10)},
  MRNUMBER = {2058264},
MRREVIEWER = {Jim\ A.\ Bryan},
       URL = {http://projecteuclid.org/euclid.jdg/1090511689},
       doi = {10.4310/jdg/1090511689},
}

@article {OP,
    AUTHOR = {Okounkov, A. and Pandharipande, R.},
     TITLE = {Hodge integrals and invariants of the unknot},
   JOURNAL = {Geom. Topol.},
  FJOURNAL = {Geometry and Topology},
    VOLUME = {8},
      YEAR = {2004},
     PAGES = {675--699},
      ISSN = {1465-3060,1364-0380},
   MRCLASS = {14H10 (14N10 57M27)},
  MRNUMBER = {2057777},
       DOI = {10.2140/gt.2004.8.675},
       URL = {https://doi.org/10.2140/gt.2004.8.675},
}

@incollection {MarinoVafa,
    AUTHOR = {Mari\~{n}o, Marcos and Vafa, Cumrun},
     TITLE = {Framed knots at large {$N$}},
 BOOKTITLE = {Orbifolds in mathematics and physics ({M}adison, {WI}, 2001)},
    SERIES = {Contemp. Math.},
    VOLUME = {310},
     PAGES = {185--204},
 PUBLISHER = {Amer. Math. Soc., Providence, RI},
      YEAR = {2002},
      ISBN = {0-8218-2990-4},
   MRCLASS = {57M27 (58J28)},
  MRNUMBER = {1950947},
       DOI = {10.1090/conm/310/05404},
       URL = {https://doi.org/10.1090/conm/310/05404},
}

@article {Kramer,
    AUTHOR = {Kramer, Reinier},
     TITLE = {K{P} hierarchy for {H}urwitz-type cohomological field
              theories},
   JOURNAL = {Commun. Number Theory Phys.},
  FJOURNAL = {Communications in Number Theory and Physics},
    VOLUME = {17},
      YEAR = {2023},
    NUMBER = {2},
     PAGES = {249--291},
      ISSN = {1931-4523,1931-4531},
   MRCLASS = {14H70 (14H10 14H81 14N35 37K20)},
  MRNUMBER = {4585764},
MRREVIEWER = {Lihua\ Wu},
       DOI = {10.4310/cntp.2023.v17.n2.a1},
       URL = {https://doi.org/10.4310/cntp.2023.v17.n2.a1},
}

@article {LLLZ,
    AUTHOR = {Li, Jun and Liu, Chiu-Chu Melissa and Liu, Kefeng and Zhou,
              Jian},
     TITLE = {A mathematical theory of the topological vertex},
   JOURNAL = {Geom. Topol.},
  FJOURNAL = {Geometry \& Topology},
    VOLUME = {13},
      YEAR = {2009},
    NUMBER = {1},
     PAGES = {527--621},
      ISSN = {1465-3060,1364-0380},
   MRCLASS = {14N35 (53D45)},
  MRNUMBER = {2469524},
MRREVIEWER = {Hsian-Hua\ Tseng},
       DOI = {10.2140/gt.2009.13.527},
       URL = {https://doi.org/10.2140/gt.2009.13.527},
}

@article {JianZhou,
    AUTHOR = {Zhou, Jian},
     TITLE = {Hodge integrals and integrable hierarchies},
   JOURNAL = {Lett. Math. Phys.},
  FJOURNAL = {Letters in Mathematical Physics},
    VOLUME = {93},
      YEAR = {2010},
    NUMBER = {1},
     PAGES = {55--71},
      ISSN = {0377-9017,1573-0530},
   MRCLASS = {14N35 (37K10 37K20)},
  MRNUMBER = {2661523},
MRREVIEWER = {Yunfeng\ Jiang},
       DOI = {10.1007/s11005-010-0397-1},
       URL = {https://doi.org/10.1007/s11005-010-0397-1},
}

\end{document}